\documentclass[12pt]{amsart}

\usepackage{amsmath, amssymb, enumerate, amscd, tikz, url, color, tikz-cd,mathtools}

\makeatletter
\@namedef{subjclassname@2020}{%
\textup{2020} Mathematics Subject Classification}
\makeatother

\theoremstyle{plain}
\newtheorem{theo}[equation]{Theorem}
\newtheorem{lem}[equation]{Lemma}
\newtheorem{cor}[equation]{Corollary}
\newtheorem{prop}[equation]{Proposition}

\newtheorem{IntroThm}{Theorem}

\theoremstyle{definition}
\newtheorem{Def}[equation]{Definition}
\newtheorem{Ex}[equation]{Example}
\newtheorem{Rem}[equation]{Remark} 
\newtheorem{Not}[equation]{Notation}

\newcommand{\C}{{\mathbb C}}
\newcommand{\F}{{\mathbb F}}
    
\newcommand{\Q}{{\mathbb Q}}

\newcommand{\Z}{{\mathbb Z}}

\newcommand{\bfc}{{\mathbf c}}

\newcommand{\bfU}{{\bf U}}

\newcommand{\Mu}{\boldsymbol \mu}

\newcommand{\ga}{{\mathfrak a}}

\newcommand{\gC}{{\mathfrak C}}

\newcommand{\gf}{{\mathfrak f}}

\newcommand{\gm}{{\mathfrak m}}

\newcommand{\go}{{\mathfrak o}}
\newcommand{\gp}{{\mathfrak p}}
\newcommand{\gP}{{\mathfrak P}}
\newcommand{\gq}{{\mathfrak q}}
\newcommand{\gQ}{{\mathfrak Q}}

\newcommand{\cA}{{\mathcal A}}
\newcommand{\cB}{{\mathcal B}}
\newcommand{\cC}{{\mathcal C}}
\newcommand{\cD}{{\mathcal D}}
\newcommand{\cE}{{\mathcal E}}
\newcommand{\cF}{{\mathcal F}}
\newcommand{\cG}{{\mathcal G}}

\newcommand{\cI}{{\mathcal I}}

\newcommand{\cM}{{\mathcal M}}

\newcommand{\cO}{{\mathcal O}}

\newcommand{\cR}{{\mathcal R}}
\newcommand{\cS}{{\mathcal S}}

\newcommand{\cU}{{\mathcal U}}
\newcommand{\cV}{{\mathcal V}}

\newcommand{\cX}{{\mathcal X}}
\newcommand{\cY}{{\mathcal Y}}
\newcommand{\cZ}{{\mathcal Z}}

\newcommand{\rD}{{\rm D}}
\newcommand{\rG}{{\rm G}}
\newcommand{\rM}{{\rm M}}

\newcommand{\rX}{{\rm X}}

\newcommand{\core}{\operatorname{core}}

\newcommand{\End}{\operatorname{End}}
\newcommand{\Ext}{\operatorname{Ext}}

\newcommand{\Gal}{\operatorname{Gal}}
\newcommand{\gr}{\operatorname{\mathbf gr}}
\newcommand{\Hom}{\operatorname{Hom}}

\newcommand{\GL}{{\rm GL}}

\newcommand{\SL}{{\rm SL}}

\newcommand{\Tr}{\operatorname{Tr}}
\newcommand{\res}{\operatorname{res}}

\newcommand{\ord}{\operatorname{ord}}

\newcommand{\un}[1]{\underline{#1}}
\newcommand{\ov}[1]{\overline{#1}}
\newcommand{\fdeg}[2]{[{#1}\!:\!{#2}]}
\newcommand{\lr}[1]{\langle{#1}\rangle}
\newcommand{\vv}[1]{\left\vert{#1}\right\vert}
 
\newcommand{\mymod}[1]{\hspace{2 pt} ({#1})}

\makeatletter
\renewcommand*\l@subsection{\@tocline{2}{0pt}{30pt}{0pt}{}}
\makeatother

\begin{document}


\baselineskip=17pt


\title[Prosaic Abelian Varieties]
{Prosaic Abelian Varieties Bad at One Prime}

\author[A. Brumer]{Armand Brumer}
\address{Department of Mathematics, Fordham University, Bronx, NY 10458, USA}
\email{brumer@fordham.edu}

\author[K. Kramer]{Kenneth Kramer}
\address{Department of Mathematics, Queens College (CUNY), Flushing, NY 11367, USA;  Department of Mathematics, The Graduate Center of CUNY, New York, NY 10016, USA}
\email{kkramer@qc.cuny.edu}

\date{20 May 2023, \, Updated 16 Sept 2025.}

\begin{abstract}
We say that an abelian variety $A_{/\Q}$ of dimension $g$ is {\em prosaic} if it is semistable, with good reduction at 2 and its points of order $2$ generate a $2$-extension of $\Q$.  

For $p \equiv 1 \mymod{8}$, let $M_u$ be the maximal 2-primary unramified abelian extension of $K = \Q(\sqrt{-p})$ and let $h_2 = \fdeg{M_u}{K}$.  We construct an indecomposable group scheme $\Xi_p$  over $\Z[\frac{1}{p}]$ of exponent 2 with field of points $M_u$.  

Assume that  $A$ is prosaic, with bad reduction at only one prime $p$.  Then $p \equiv 1 \mymod{8}$ and $A$ is totally toroidal at $p$.  We prove that if $\End A=\Z$, then there is a $\Q$-isogenous abelian variety $B$ such that $B[2]$ is a subquotient of $\Xi_p$.  We thereby show that $2g+2 \le h_2$ and $p$ has the form $a^2+16b^2$, with $a+4b \equiv \pm 1 \mymod{8}$.  Moreover, if $2g + 4 \le h_2$, then $p$ has the form $a^2+64b^2$, with $a \equiv \pm 1 \mymod{8}$.
\end{abstract}



\maketitle

\tableofcontents

\numberwithin{equation}{section}

\section{Introduction}
Neumann \cite{Neu} and Setzer \cite{Set} independently showed that there are elliptic curves over $\Q$ of prime conductor $p$ whose points of order 2 generate a 2-extension of $\Q$ if and only if $p = 17$ or $p$ has the form $p = u^2+64$ for some integer $u$.   Curiosity about higher dimension led us to the following definition. 

Let $A_{/\Q}$ be an abelian variety with good reduction at 2 and let $A[2]$ be the kernel of multiplication by 2.  We say that $A$ is {\em prosaic} if it is {\em semistable}, with {\em good reduction at $2$} and the 2-division field $\Q(A[2])$ is a 2-{\em extension} of   $\Q$.  These properties are preserved  by taking products or quotients and under isogeny by Proposition \ref{IsogenyInvariant}.  We focus on abelian varieties $A_{/\Q}$ such that 
\begin{equation}\label{AIntro}
\fbox{\parbox{340 pt}{ $A \text{ is simple, prosaic, has good reduction outside } p \text{ and} \, \dim A = g. $}}
\end{equation}

 Our work owes much to that of  Fontaine, Abrashkin and Schoof, who classified semistable abelian varieties with good reduction outside the set of primes dividing a small integer $N$.  In several instances, they used the Odlyzko discriminant bounds to show that the abelian varieties they considered are prosaic. 

To study putative abelian varieties $A$ satisfying \eqref{AIntro}, we require some number theory of $K = \Q(\sqrt{-p})$, reviewed in \S \ref{Arith}.  Let $M_u$ be the maximal unramified 2-primary extension of $K$ and let $h_2 = \fdeg{M_u}{K}$.  

In \S \ref{GrpSchms}, we define a full subcategory $\cA$ of the category $\un{Gr}$ of finite flat 2-primary group schemes over $\Z[\frac{1}{p}]$ whose objects include the subquotients of $A[2^n]$ for all $n$.  In addition to the constant group scheme $\cZ_2 = \Z/2\Z$ and its dual $\Mu_2$, the category $\cA$ contains certain group schemes $\Phi$ and $G_{-1}$ of exponent 2, representing non-trivial extension classes of $\Mu_2$ by $\cZ_2$ and of $\cZ_2$ by $\Mu_2$, respectively.  

In \S \ref{dPhiSection}, we introduce a group scheme $\Xi_p$ of exponent 2 in $\cA$ whose Galois module is indecomposable and cyclic over the group ring $\F_2[G_\Q].$  The group scheme $\Xi_p$ has a filtration with successive quotients isomorphic to $G_{-1}$, field of points  $M_u$ and $\F_2$-dimension $h_2$.  

Define a $(d,\Phi)$-bloc to be an object of exponent 2 in $\cA$ admitting a filtration with $d$ successive quotients isomorphic to $\Phi$, whose Galois module is cyclic.  Certain subquotients of $\Xi_p$ are examples of $(d,\Phi)$-blocs.  The following consequence of Theorem \ref{dPhiThm} provides necessary and sufficient conditions for existence of a $(d,\Phi)$-bloc.  Theorem \ref{dPhiThm} also gives additional constraints on $h_2$ and the prime $p$.

\begin{IntroThm} 
There is a $(d,\Phi)$-bloc in $\cA$ if and only if $2d+2 \le h_2$.
\end{IntroThm}

In \S \ref{AbVar}, our results about group schemes are applied to abelian varieties $A$ satisfying \eqref{AIntro}.  Then Theorems \ref{1Mod8} and \ref{A[2]bloc} provide the following information.

\begin{IntroThm}  \label{IntroAThm}
If $A$ sastisfies {\rm \eqref{AIntro}}, $g \ge 2$ and $\End A = \Z$, then: 
\begin{enumerate}[{\rm i)}]
\item $A$ is totally toroidal at $p$; that is, the conductor of $A$ is $p^g$.  \vspace{2 pt}
\item There is a $\Q$-isogenous abelian variety $B$ such that $B[2]$ is $(g,\Phi)$-bloc. \vspace{2 pt}
\item $2g+2 \le h_2$ and $p$ has the form $a^2+16b^2$, with $a+4b \equiv \pm 1 \mymod{8}$.
\end{enumerate}
If $g=2$ and more generally, if $2g+4 \le h_2$, then $p$ has the form $a^2+64b^2$ with $a \equiv \pm 1 \mymod{8}$.
\end{IntroThm}
 
We do not know of any abelian variety $A$ satisfying \eqref{AIntro} when $g \ge 2$ and $\End  A = \Z$.  As a consequence of Theorem \ref{IntroAThm} and Remark \ref{P1*}, the non-existence of such abelian surfaces is verified for a set of  primes $p$ of asymptotic density 15/16.  Below 1000, only  the status of $p= 113$, 257, 337, 353, 577, 593 and 881 is unknown.  Even assuming our paramodular conjecture, searching for the corresponding paramodular forms of level $p^2$ does not seem  currently feasible.

In contrast, Ribet \cite{Rib3} proved that the modularity of RM abelian varieties follows from Serre's conjecture.  As a result, we verified that for $2 \le g \le 20$ and $p<10^6$, there are only three RM abelian varieties up to isogeny satisfying \eqref{AIntro}.  Details are given in Remark \ref{RMExamples}.  In \cite{BK4}, somewhat different methods are used to show that the conclusions of Theorem \ref{IntroAThm} remain valid for RM abelian varieties satisfying \eqref{AIntro}.

If $A$ is a prosaic abelian variety of dimension $g$ and odd conductor $N$, then $N$ is the product of at least $g$ primes counting multiplicities, by \cite[Cor.\! 5.4]{BK3}.  It is fairly easy to produce examples  if one allows more than one place of bad reduction. In \cite{BK4}, we exhibit six two-parameter families of prosaic abelian surfaces with $\End  A = \Z$.  Under the Schinzel hypothesis, they yield infinitely many prosaic Jacobians of conductor $pq$, for primes $p \ne q$.  

Fix the completion $\ov{\Q}_2$ of an algebraic closure of $\Q_2$, which contains all the algebraic extensions of $\Q$ or $\Q_2$ considered here.

\numberwithin{equation}{section}
\section{Input from higher arithmetic} \label{Arith}

Throughout  this section, $p$ is prime and $p \equiv 1 \mymod{8}$.  Then $p$ is a sum of squares and $p \equiv 1 \mymod{8}$ is equivalent to $p = a^2+16b^2$ for some integers $a,b$ by congruences modulo 8.  Unique factorization in $\Z[i]$ implies that $a$ and $b$ are determined up to sign and so the following sets are well-defined:
\begin{equation}  \label{P1P3}
\begin{array}{l l}
\text{\bf P1} = \{\text{primes } p \equiv 1 \mymod{8} \mid a+4b \equiv \pm 1 \mymod{8} \}, \\[3 pt]
 \text{\bf P3} = \{\text{primes } p \equiv 1 \mymod{8} \mid a+4b \equiv \pm 3 \mymod{8} \}.
\end{array}
\end{equation} 
Let $\cO_K$ be the ring of integers of $K = \Q(\sqrt{-p})$.  Write $H_2$ for the 2-Sylow subgroup of the ideal class group of $K$ and $h_2 = h_2(-p)$ for its order.   According to \cite{RR,BC}:
\begin{equation} \label{h2}
h_2 = 4 \text{ if } p \text{ is in {\bf P3}} \, \text{ and } \, 8 \vert{h_2} \text{ if } p \text { is in {\bf P1}}.
\end{equation}

Conventions of \cite[IV]{Ser} will be used for the numbering of higher ramification groups in the following local situation.  Let $\tilde{L}$ be a finite extension of $\Q_2$ with ring of integers $\tilde{\cO}$ and prime ideal $\tilde{\lambda}$.  Suppose that $\tilde{L} \supseteq \tilde{K} \supseteq \Q_2$, with $\tilde{L}$ Galois over $\tilde{K}$ and $G = \Gal(\tilde{L}/\tilde{K})$.  For  $j \ge 0$, the lower numbering is given by
$$
G_j = \{g \in G \mid g(x) - x \in \tilde{\lambda}^{j+1} \text{ for all } x \in \tilde{\cO} \}.
$$ 
In the upper numbering, $G^{\varphi(j)} = G_j$, using the Herbrand function:
$$
\varphi(u) = \int_0^u \frac{dt}{\fdeg{G_0}{G_t}} \, .
$$
In particular, $G_0 = G^0 = \cI_{\tilde{\lambda}}(\tilde{L}/\tilde{K})$ is the inertia subgroup of $G$.  For $j \ge 1$, the successive quotients $G_j/G_{j+1}$ are elementary abelian $2$-groups, possibly trivial.

\begin{prop} \label{RayClass1a}
Let $M$ be the maximal $2$-primary abelian extension of $K = \Q(\sqrt{-p})$ with ray class modulus $2 \cO_K$ and let $h_2$ be the $2$-primary factor of the class number of $K$.  Then $\Gal(M/K)$ is cyclic of order $2h_2$.  At places $\lambda_M$ over $2$ in $M$, the ramification group $\cI_{\lambda_M}(M/K)_2 = \{1\}$ in the lower numbering and 
$
\cI_{\lambda_M}(M/K)^{1+\epsilon} = \{1\}
$
for every $\epsilon > 0$ in the upper numbering.
\end{prop}

\begin{proof}
Denote the idele group of $K$ by $J_K$. Let $U_w$ be group of units in the completion $K_w$ at a non-archimedean prime $w$ and let $\lambda_K$ be the prime over 2 in $K$.  Let $F$ be the maximal unramified abelian extension of $K$, i.e., the Hilbert class field of $K$.  Then $F$ is a subfield of the abelian extension $L$ of $K$ with modulus $2\cO_K = \lambda_K^2$ and the norm group $N_{L/K}(J_L)$ contains at least the following local components:
\begin{enumerate}[\hspace{2 pt} $\cdot$]
\item the unit group $U_w$ for non-archimedean primes $w \ne \lambda_K$,  \vspace{2 pt}

\item $U_{\lambda_K}^{(2)} = \{ u \in U_{\lambda_K} \mid  u \equiv 1 \bmod{\lambda_K^2} \}$, with $U_{\lambda_K} \simeq \Z_2[i]^\times$, \vspace{2 pt} 

\item $\C^\times$ at the archimedean place of $K$.
\end{enumerate}
Let 
$
\bfU = U_{\lambda_K} \times \C^\times \times \prod_{w \ne \lambda_K} U_w
$
be the group of unit ideles in $J_K$.  Class field theory gives the vertical isomorphisms in the diagram:
\begin{equation} \label{cft1}
\begin{tikzcd}
J_K/(U_{\lambda_K}^{(2)} \times \C^\times \times \prod_{w \ne \lambda_K} U_w) K^\times  \ar["\psi"  , r] \ar["\simeq", d] &J_K/\bfU  K^\times  \ar["\simeq", d]  \\
\Gal(L/K) \ar["{\rm restriction}",r]  & \Gal(F/K). 
\end{tikzcd}
\end{equation}
Surjectivity of $\psi$ is clear, as is the surjection
$$
j\!: \, U_{\lambda_K}/U_{\lambda_K}^{(2)} \to \ker \psi = \bfU \, K^\times/(U_{\lambda_K}^{(2)} \times \C^\times \times \prod_{w \ne \lambda_K} U_w) \, K^\times
$$
induced by $a \mapsto (a,1,1, \dots)$ for $a$ in $U_{\lambda_K}$.  The coset $aU_{\lambda_K}^{(2)}$ is in $\ker j$ exactly when
$$
(a,1,1, \dots) = (b,c, \prod_{w \ne \lambda_K} u_w)k
$$
with $b \in U_{\lambda_K}^{(2)}$, $c \in \C^\times$, $u_w \in U_w$ and $k \in K^\times$.  Then $k$ is a global unit in $K$, so $k = \pm 1$.  Thus $aU_{\lambda_K}^{(2)}$ is in $\ker j$ if and only if $a = \pm b$.  But $-1 = 1-2$ is in $U_{\lambda_K}^{(2)}$, since the ideal $\lambda_K^2 = 2\cO_K$ and so $a$ is in $U_{\lambda_K}^{(2)}$.   Hence $j$ is injective and gives an isomorphism $U_{\lambda_K}/U_{\lambda_K}^{(2)} \simeq \ker \psi$.

The residue field $\cO_K/\lambda_K \simeq \F_2$, so $\fdeg{U_K}{U_{\lambda_K}^{(2)}} = 2$.  Hence $\fdeg{L}{K} = 2h$, where $h = \fdeg{F}{K}$ is the class number of $K$.  If $M_u$ is the maximal 2-primary extension of $K$ in $F$, then $\fdeg{M_u}{K} = h_2$ and so $\fdeg{M}{K} = 2h_2$.  Recall that by genus theory, the maximal {\em elementary} 2-extension of $K$ in $M_u$ is cyclic, given by $M'_u = \Q(i,\sqrt{p})$, so  $M_u/K$ also is cyclic.  

To establish that $M/K$ is cyclic, we will show that the maximal elementary 2-extension $M'$ of $K$ in $M$ is $M'_u$.  Because $p \equiv 1 \mymod{8}$, there is an isomorphism of $K_{\lambda_K}$ to $\Q_2(i)$ and $1-i$ is a prime element in $\Q_2(i)$.  Then $U_{\lambda_K}/U_{\lambda_K}^{(2)}$ is generated by the coset of 
$$
\alpha = 1+(1-i) = 2-i = \frac{1}{2} \, (1+2i) \, (1-i)^2 \in \frac{1}{2} \, U_{\lambda_K}^{(2)} \, K_{\lambda_K}^{\times 2}.
$$ 
In $J_K$, we therefore have
$$
\begin{array}{l l l}
(\alpha, 1, 1, \dots) &=& ( \, (1+2i)(1-i)^2, 1, 1,\dots) \, (1, 2, 2, \dots) \, (\frac{1}{2}, \frac{1}{2}, \dots) \\[3 pt] 
&\in& (U_{\lambda_K}^{(2)} K_{\lambda_K}^{\times 2} , 1, 1, \dots)  \, (\{1\} \times \C^\times \times \prod_{w \ne {\lambda_K}} U_w) \,K^\times) \\[3 pt]
&\in& \left(U_{\lambda_K}^{(2)} \times \C^\times \times \prod_{w \ne {\lambda_K}} U_w \right) J_K^2 \, K^\times.
\end{array}
$$
Hence the image of $j$ becomes trivial modulo $J_K^2$.  It follows that if we reduce the top row of \eqref{cft1} modulo $J_K^2$, then $\psi$ induces an isomorphism $\Gal(M'/K) \simeq \Gal(M'_u/K)$.   Since $M'_u$ is contained in $M'$, we find that
\begin{equation} \label{M'}
M' = M'_u = \Q(i,\sqrt{p})
\end{equation}
is quadratic over $K$ and so $M/K$ is cyclic.

In the upper numbering, the ramification group $\cI_{\lambda_M}(M/K)^{1+\epsilon} = \{1\}$ for every $\epsilon > 0$, since the ray class conductor exponent $\gf_{\lambda_M}(M/K) = 2$.  In particular, let $\epsilon_0 = \vv{\cI_{\lambda_M}(M/K)_2}/\vv{\cI_{\lambda_M}(M/K)_1}$.  Then the Herbrand function $\varphi(2) = 1+\epsilon_0$, so $\cI_{\lambda_M}(M/K)_2 = \cI_{\lambda_M}(M/K)^{1+\epsilon_0} = \{1\}$.
\end{proof}

As in Proposition \ref{RayClass1a}, let $M$ be the maximal $2$-primary abelian extension of $K$ with ray class modulus $2 \cO_K$, let $M_u$ be the maximal subfield of $M$ unramified over $K$ and let $M_s$ be the maximal subfield of $M_u$ in which $\lambda_K$ splits completely.  By maximality, $M$, $M_u$ and $M_s$ are Galois over $\Q$.  Fix a prime $\lambda_M$ over $\lambda_K$ in $M$ and write $\lambda_u$ and $\lambda_s$ for the primes below it in $M_u$ and $M_s$.  Fix $v$ over $p$ in $M$. Additional properties of $M/K$ are addressed in the following Corollary.

\begin{cor}  \label{RayClass1b}
The ideal $2\cO_K = \lambda_K^2$ and $\lambda_K$ factors as follows:
\begin{equation} \label{fieldtower}
\Q \hookrightarrow K \xhookrightarrow{\rm split} M_s  \xhookrightarrow{\rm inert} M_u  \xhookrightarrow{\rm ramified} M,
\end{equation}
with $\fdeg{M_u}{M_s} = 2$, $\fdeg{M}{M_u} = 2$ and $\Gal(M/K) \simeq \Z/2^{2h_2}\Z$.
The following properties hold.
\begin{enumerate}[{\rm  i)}]
\item  The inertia group $\cI_v(M/\Q) = \lr{\sigma_v}$ at $v$ has order $2$.  \vspace{2 pt}
\item  The inertia group $\cI_{\lambda_M}(M/\Q)$ contains an involution $\sigma_{\lambda_M}$ satisfying $\sigma_{\lambda_M}(i) = -i$ and $\sigma_{\lambda_M}(\sqrt{p}) = \sqrt{p}$.    
\vspace{2 pt}
\item $\Gal(M/\Q) = \lr{\sigma_v, \sigma_{\lambda_M}}$ is isomorphic to the dihedral group of order $4h_2$, with cyclic subgroup $\Gal(M/K)$ of order $2h_2$ generated by $\tau= \sigma_{\lambda_M}  \sigma_v$.  \vspace{2 pt} 
\item The maximal subfield of $M$ abelian over $\Q$ is $\Q(i, \sqrt{p})$.   \vspace{2 pt}
\item  The Frobenius $\phi_{\lambda_u}$ at $\lambda_u$ is the unique involution in $\Gal(M_u/K)$ and the residue degree of $\lambda_M$ in $M/\Q$ is $2$. 
\end{enumerate}
\end{cor} 

\begin{proof}
Since $v$ is unramified in $M/K$, the inertia group $\cI_v(M/\Q)$ has order 2, so item (i) holds.  To ease notation, let $\lambda = \lambda_M$.

By Proposition \ref{RayClass1a}, the ramification index of ${\lambda}$ in $M/K$ is 2, so the inertia group $\cI_{\lambda}(M/\Q)$ has order 4 and equals the first ramification group  $\cI_{\lambda}(M/\Q)_1$.  Furthermore, it contains an element $\sigma_{\lambda}$ such that $\sigma_{\lambda}(i) = -i$ and $\sigma_{\lambda}(\sqrt{p}) = \sqrt{p}$.  Since successive quotients in the lower numbering are elementary 2-groups, $\sigma_{\lambda}^2$ is in $\cI_{\lambda}(M/\Q)_2$.  But $\sigma_{\lambda}^2$ also is trivial on $K$, so it is in 
$$
 \cI_{\lambda}(M/\Q)_2 \cap \Gal(M/K) = \cI_{\lambda}(M/K)_2 = \{1\}.
$$
Hence, $\sigma_{\lambda}$ is an involution.  This verifies (ii).

Let $M''$ be the maximal elementary 2-extension of $\Q$ inside $M$.  Then   $M''$ is contained $M'$ by definition of $M'$.  By \eqref{M'}, $M' = \Q(i,\sqrt{p})$ and so $M'' = M'$.  Since $\Gal(M/\Q)$ is a 2-group and the restrictions of $\sigma_{\lambda}$ and $\sigma_v$  generate $\Gal(M''/\Q)$, the Burnside basis theorem \cite[Thm. 12.2.1]{Ha} implies that $\Gal(M/\Q) = \lr{\sigma_{\lambda},\sigma_v}$. In addition, $\Gal(M/\Q) = \lr{\sigma_{\lambda}, \tau}$, where $\tau = \sigma_{\lambda}  \sigma_v$.  The relation
$$
\sigma_{\lambda}^{-1} \tau \sigma_{\lambda} = \sigma_v \sigma_{\lambda} = \tau^{-1},
$$
shows that $\Gal(M/\Q)$ is dihedral and so (iii) holds.  It also follows that $M''$ is the maximal subfield of $M$ abelian over $\Q$, so (iv) holds.

To complete the factorization over 2 in \eqref{fieldtower}, note that $\lambda_K$ is not principal in $\cO_K$, but $\lambda_K^2 = 2\cO_K$ and so a Frobenius $\phi = \phi_{\lambda_u}$ in the decomposition group $\cD_{\lambda_u}(M_u/K)$ has order 2.  Thus, $\phi$ is the unique involution in the cyclic group $\Gal(M_u/K)$ and $\fdeg{M_u}{M_s} = 2$.  It follows that the residue degree of $\lambda_M$ over $\Q$ also is 2.
\end{proof}

\begin{Rem}  \label{P1*}
In \S \ref{AbVar}, the following subset of {\bf P1} is used:
\begin{equation} \label{RemP1*}
\begin{array}{lll}
{\bf P1^*} &=& \{ \text{primes } p \text{ of the form } a^2+64c^2 \mid a \equiv \pm 1 \mymod{8}\} \\[3 pt]
                &=& \{ p \in {\bf P1} \mid p \equiv 1 \mymod{16}\}.
\end{array}
\end{equation}
It is well-known \cite{BC} that primes $p \equiv 1 \mymod{8}$ are in {\bf P1} if and only if they split completely in $\Q(\zeta_8, \sqrt{1-i})/\Q$ and so they have asymptotic density 1/8.  A prime $p$ in {\bf P1} belongs to ${\bf P1^*}$ if and only if it also satisfies $p \equiv 1 \mymod{16}$.  Equivalently, $p$ splits completely in $\Q(\zeta_{16},\sqrt{1-i})/\Q$.  Hence, asymptotically among primes in {\bf P1}, half are in ${\bf P1^*}$.
\end{Rem}

\numberwithin{equation}{section}

\section{Some $2$-primary group schemes and their extensions} \label{GrpSchms}

Throughout  this section, $p$ is prime and $p \equiv 1 \mymod{8}$.    Let $\un{Gr}$ be the category of finite flat 2-primary commutative group schemes $V$ over $\Z[\frac{1}{p}]$.  Write $V(\ov{\Q})$ for the Galois module of $V$ and $\Q(V)$ for the field generated by the points of $V(\ov{\Q})$.  Let the {\em prosaic category} $\cA$ be the full subcategory of $\un{Gr}$ whose objects satisfy the following conditions:
\begin{enumerate}
\item[$\cA1$.] The inertia group $\cI_v(\Q(V)/\Q) = \lr{\sigma_v}$ is cyclic at each place $v$ over $p$ and $(\sigma_v-1)^2(V(\ov{\Q}) )= 0$. \vspace{2 pt}
\item[$\cA2$.] The Galois group $\Gal(\Q(V)/\Q)$ is a 2-group.   \vspace{2 pt}
\end{enumerate}

Condition $\cA1$ reflects Grothendieck's  criterion for semistability \cite[Prop.\! 5.13, p.\! 70]{Gro}.  It follows from the definition of a prosaic abelian variety $A_{/\Q}$ that $A[2^n]$ is an object in $\cA$ for all $n \ge 1$.

Schoof \cite[pp. 5]{Sch4} describes the full subcategory $\un{C}$ of $\un{Gr}$ whose objects satisfy $\cA1$.  (Also see \cite[\S2]{Sch2}, where the relevant category is called $\un{D}$.)   Closed  finite flat subgroup schemes of objects of $\un{C}$ also are in $\un{C}$, as are quotients by such subgroup schemes and binary products.  If $G$ is an object of $\un{C}$, then the Cartier dual $G^\vee$ is again an object of $\un{C}$.  If $G$ and $G'$ are in $\un{C}$, then the group $\Ext^1_{\un{Gr}}(G,G')$ classifies extensions of $G$ by $G'$ in the category $\un{Gr}$, with the group operation being Baer sum.  The subset $\Ext^1_{\un{C}}(G,G')$ of $\Ext^1_{\un{Gr}}(G,G')$ comprised of extension classes that are represented by short exact sequences of objects in $\un{C}$ is a subgroup \cite[\S2.3]{Sch2}.

Let $V_1$ and $V_2$ be objects of $\cA$ and let the exact sequence
\begin{equation} \label{ExtensionClass}   
0 \to V_1 \to V \to V_2 \to 0
\end{equation}
represent a class in $\Ext^1_{\un{C}}(V_1,V_2)$.  Since $\Gal(\Q(V)/\Q)$ necessarily is a 2-group, \eqref{ExtensionClass} also represents a class in $\Ext^1_\cA(V_1,V_2)$.  It follows that $$
\Ext^1_\cA(V_1,V_2)= \Ext^1_{\un{C}}(V_1,V_2).
$$  
Associated to a short exact sequence in $\un{C}$, one has the long exact sequences of Hom-Ext \cite[\S8]{Sch2}. It follows that the same is true for the category $\cA$.  If $W$ is in $\cA$, then \eqref{ExtensionClass} leads to the exact sequence   \vspace{2 pt}
\begin{equation*}
\begin{array}{l}  
\Hom_{\cA}(W,V_2)  \leftarrow \Hom_{\cA}(W,V) \leftarrow \Hom_{\cA}(W,V_1) \leftarrow 0 \\
 \hspace{35 pt}  \downarrow  \\
   \Ext^1_{\cA}(W,V_1)   \to  \Ext^1_{\cA}(W,V) \to \Ext^1_{\cA}(W,V_2) \to \dots 
\end{array}
\end{equation*}
for $\Hom_\cA(W,-)$.  The contravariant sequence for $\Hom_\cA(-,W)$ also holds.

If $2V_1 = 2V_2 = 0$, let $\Ext^1_{\cA,[2]}(V_1,V_2)$ be the subgroup of $\Ext^1_\cA(V_1,V_2)$ comprising extension classes represented by \eqref{ExtensionClass} with $2V = 0$.  

\begin{Not}
A {\em flag} of group schemes is an increasing filtration 
\begin{equation} \label{flag}
\cF\!:  \quad 0 = V_0 \subset V_1 \subset \cdots \subset V_{n-1} \subset V_n = V
\end{equation}
of objects in $\cA$ such that the inclusions are closed immersions $V_{j-1} \hookrightarrow V_j$ for $j = 1, \dots n$.  Write $\gr(\cF) = [V_1/V_0, \dots, V_n/V_{n-1}]$ for the list of successive quotients.  We may specify $\gr(\cF)$, without making the filtration $\cF$ explicit.
\end{Not}  

Next, we recall a result of Raynaud \cite[\S2.1]{Ray} that holds for finite flat group schemes $\cG$ over a Dedekind domain $S$ with field of fractions $F$.  Also see \cite[Ex. 7.4.4]{BrCo},  \cite[Lemma, p. 45]{Vo},  \cite[Lemma 1.1.1]{Con}.  

\begin{lem} \label{GalToGrpSch}
Let $G = \cG \otimes_{\cS} F$ be the generic fiber of $\cG$ and let $G_1$ be a subquotient of $G$; i.e., $G_1 = G''/G'$ for closed immersions 
$
G' \hookrightarrow G'' \hookrightarrow G
$
of finite flat group schemes over $F$.  Then there are unique closed immersions 
$
\cG' \hookrightarrow \cG'' \hookrightarrow \cG
$
of finite flat group schemes over $\cS$ such that 
$$
\cG'' \otimes_S F = G'', \hspace{10 pt} \cG' \otimes_S F = G'
$$ 
and a unique isomorphism $(\cG''/\cG') \otimes_S F \simeq G_1$, compatible with $G''/G' = G_1$.
\end{lem}

\begin{prop} \label{ZMuFil}
The only simple objects in $\cA$ are the \'etale group scheme $\cZ_2 = \Z/2\Z$ and its Cartier dual, $\Mu_2$.  A commutative finite flat $2$-primary group scheme $V$ over $\Z[\frac{1}{p}]$ satisfies $\cA2$ if and only if it has a filtration by closed subgroup schemes with successive quotients isomorphic to $\Mu_2$ or $\cZ_2$.  
\end{prop}

\begin{proof}
If $V$ has such a filtration, then $\Gal(\Q(V)/\Q)$ certainly is a 2-group.  For the converse, recall that there always is a non-trivial fixed point for the action of a 2-group on a finite dimensional $\F_2$-vector space.  Hence the Galois module $V(\ov{\Q})$ admits an increasing filtration of Galois submodules
$$
0 = W_0 \subset W_1 \subset \dots \subset W_n = V(\ov{\Q})
$$ 
in which $\vv{W_j/W_{j-1}} = 2$ for $j = 1, \dots, n$.  Let $V_j$ be the closed finite flat subgroup scheme of $V$ corresponding to $W_j$ by Lemma \ref{GalToGrpSch}.  Then $V_j/V_{j-1}$ is isomorphic to $\Mu_2$ or $\cZ_2$ by the classification of Oort-Tate \cite{TO}.
\end{proof}

\begin{lem} \cite[Prop.~2.3]{Sch1}.     \label{Patch}  
Let $R = \Z[\frac{1}{p}]$ and $R' = \Z[\frac{1}{2p}]$.  Let $\cB$ be the category of triples $(U,W,\theta)$, where $U$ is a finite flat group scheme over $\Z_2$, $W$ is a finite flat group scheme over $R'$ and 
$$
\theta\!: \,U \otimes_{\Z_2} \Q_2 \to W \otimes_{R'} \Q_2
$$ 
is an isomorphism of group schemes over $\Q_2$.  Then the functor 
$$
\un{Gr} \to \cB \quad \text{by} \quad V \mapsto (V\otimes_R \Z_2,V\otimes_R R', id \otimes_R \Q_2)
$$ 
is an equivalence of categories. For objects $V$ and $V'$ of $\un{Gr}$ to be isomorphic it is sufficient that:
\begin{enumerate}[{\rm i)}]
\item $V$ and $V'$ are isomorphic over $\Z_2$ and over $R'$ and
\item the map $\Hom_{\Z_2}(V,V') \to \Hom_{\Q_2}(V,V')$ is surjective.\end{enumerate}
\end{lem}

In the lemma above, $V \otimes_R R'$ is \'{e}tale, so it can be identified with the Galois module $V(\ov{\Q})$ associated to $V$.  From now on, {\em group scheme} should be taken to mean {\em finite flat commutative group scheme in $\cA$}, unless otherwise indicated.   Note that the additional requirements for objects $V$ of $\un{Gr}$ to be in $\cA$ are conditions on $V(\ov{\Q})$. 

\begin{Not} \label{Vm}
Let $V$ be a group scheme in $\cA$ and let $\widetilde{V} = V \otimes_R \, \Z_2$.  Denote the multiplicative subgroup scheme and the \'{e}tale quotient of $\widetilde{V}$ by $\widetilde{V}^m$ and $\widetilde{V}^{et}$, respectively.  Let  $L_\lambda$ be the completion of $L = \Q(V)$ at a fixed prime $\lambda$ over 2 and let $\cD_\lambda(L/\Q) \simeq \Gal(L_\lambda/\Q_2)$ be the decomposition group of $\lambda$ inside $\Gal(L/\Q)$.  Let $V^m(L,\lambda)$ be the $\cD_\lambda(L/\Q)$-submodule of the global Galois module $V(L)$ corresponding to $\widetilde{V}^m$.  Note that the points of $V^m(L,\lambda)$ are in $L$.  By assumption, $\Gal(L/\Q)$ is a 2-group, so the connected component $\widetilde{V}^0 = \widetilde{V}^m$.  Hence $\widetilde{V}^{et} = \widetilde{V}/\widetilde{V}^m$ and  $V^{et}(L,\lambda) = V(L)/V^m(L,\lambda)$ as a $\cD_\lambda(L/\Q)$-module. 
\end{Not}

\begin{lem} \label{InertiaOnVm}
Let $V$ be an object of $\cA$ and $L = \Q(V)$.  If $\sigma$ is an element of the inertia group $\cI_\lambda(L/\Q)$ at $\lambda$ over $2$, then 
\begin{enumerate}[{\, \rm i)}]
\item $\sigma$ acts trivially on $V^{et}(L,\lambda)$ and $(\sigma-1)(V(L)) \subseteq V^m(L,\lambda)$.
\item $\sigma$ acts on $V^m(L,\lambda)$ via the cyclotomic character.
\end{enumerate}
\end{lem}

\begin{proof}
Item (i) follows from trivial action of the inertia group on the points of an \'{e}tale group scheme over $\Z_2$.  For item (ii), let $W$ be the Cartier dual of $V^m$.  Then $W$ is \'{e}tale and gives a perfect $\cD_\lambda(L/\Q)$-equivariant pairing $V^m(L,\lambda) \times W(L_\lambda) \to \Mu_{2^n}$, where $2^n$ is the exponent of $V^m(L,\lambda)$.  Since $\sigma$ is trivial on $W(L,\lambda)$, it acts on $V^m(L,\lambda)$ as it acts on $\Mu_{2^n}$.
\end{proof}

The field of points $L = \Q(V)$ of an object $V$ of $\cA$ is unramified outside $\Sigma = \{2, p, \infty\}$.   As in \S\ref{Arith}, we use the conventions of \cite[IV]{Ser} for the upper and lower numbering of higher ramification subgroups in the inertia group $\cI_\lambda(L/\Q)$ at the fixed place $\lambda$ over 2.  In Serre's notation, if the exponent of $V$ is $2^n$, then Fontaine \cite{Fon} asserts that 
\begin{equation} \label{FontBd}
\cI_\lambda(L/\Q)^u  = 1 \text{ for all } u > n.
\end{equation}

\vspace{5 pt}

Recall that $p \equiv 1 \mymod{8}$ and that $M$ denotes the maximal 2-primary abelian extension of $K  = \Q(\sqrt{-p})$ of conductor $2\cO_K$.   See Corollary \ref{RayClass1b} for the arithmetic of $M$ and additional notation used below.

\begin{prop}   \label{FGField}
If $V$ is an object of $\cA$ and $2 V = 0$, then $L = \Q(V)$ is contained in $M$. 
\end{prop}

\begin{proof}
It suffices to replace $L = \Q(V)$ by $L = K(V)$.   Property $\cA2$ of the category $\cA$ implies that $\Gal(L/\Q)$ is a 2-group.  Let $\gp$ be the unique prime over $p$ in $K$ and let $\gP$ be a prime of $L$ over $\gp$.  Since $2V = 0$, the inertia group $\cI_\gP(L/\Q)$ has order 2 by $\cA1$.  But $p$ ramifies in $K$, so $L/K$ is unramified over $\gp$ and $L
/K$ might ramify only at primes $\lambda$ of $L$ over 2.  For subfields of $L$, we use a tilde to denote the completion at the prime below $\lambda$.

Next, we evaluate the ray class conductor $\bfc$ of $L'/K$.  The bound \eqref{FontBd} on higher ramification for $\tilde{L}'/\Q_2$ holds with $n = 1$.  Then, the conductor computation in \cite[Lemma 6]{BK1}, applied to the sequence of inclusions
$$
\Q_2 \hookrightarrow \tilde{K} = \Q_2(i) \hookrightarrow \tilde{L}'\hookrightarrow \tilde{L} = \tilde{K}(V),
$$
shows that $\bfc$ divides $(1+i)^2 = 2\cO_K$ and thus $L'$ is contained in $M$.  But $\Gal(M/K)$ is cyclic by Proposition \ref{RayClass1a}, so $\Gal(L'/K)$ is cyclic.  Since $\Gal(L/K)$ is a 2-group with cyclic abelianization, $\Gal(L/K)$ also is cyclic by the Burnside basis theorem.  Hence $L = L'$ is contained in $M$.
\end{proof}

Recall that if $V_1$ and $V_2$ are in $\cA$, then $\Ext^1_{\cA}(V_1,V_2) = \Ext^1_{\un{C}}(V_1,V_2)$.  

\begin{Not} \label{NotPhi}
Since $p \equiv 1 \mymod{8}$, \cite[Cor.~4.2]{Sch2} implies that 
$$
\dim_{\F_2} \Ext^1_{\cA}(\Mu_2,\cZ_2) = 1.
$$  
Let $\Phi$ represent the non-trivial class in $\Ext^1_{\cA}(\Mu_2,\cZ_2)$.  
\end{Not} 

See \cite[Prop. 2.4]{Sch4} for basic properties of $\Phi$.  In particular, $2\Phi = 0$ and $\Phi$ is Cartier self-dual.  The points of $\Phi$ generate the field $\Q(\Phi) = \Q(\sqrt{p})$. The generator $\sigma_v$ for the inertia group $\cI_v(M/\Q)$ in Corollary \ref{RayClass1b} at $v$ over $p$ generates $\Gal(\Q(\sqrt{p})/\Q)$ and is represented by an upper triangular unipotent matrix with respect to the flag $0 \subset \cZ_2 \subset \Phi$.  Over $\Z_2$, there is a splitting $\Phi \otimes_R \Z_2 \simeq \Mu_2 \oplus \cZ_2$. 

\begin{lem} \label{bigExt}
If $W$ has a filtration $\cF$ with $\gr(\cF)$ given by $d$ copies of $\Phi$, then $\Ext^1_{\cA}(\Mu_2,W) = 0$.  If $W'$ has a filtration $\cF'$ with successive quotients isomorphic to $\cZ_2$, then $\Ext^1_{\cA}(\Phi,W') = 0$.
\end{lem}

\begin{proof} The case  $d=1$, namely $\Ext^1_{\cA}(\Mu_2,\Phi) = 0$,
 is in \cite[Prop.\!~2.5]{Sch4}.  If $d \ge 2$, the filtration $\cF$ leads to a short exact sequence
$$
0 \to \Phi \to W \to W' \to 0
$$ 
in which $W'$ has a filtration with $d-1$ successive quotients isomorphic to $\Phi$. Applying $\Hom(\Mu_2,-)$ gives exactness of
$$
\to \Ext^1_{\cA}(\Mu_2,\Phi) \to \Ext^1_{\cA}(\Mu_2,W) \to \Ext^1_{\cA}(\Mu_2,W') \to
$$  
The first group is trivial because $d=1$ and the last vanishes by induction, so $\Ext^1_{\cA}(\Mu_2,W) = 0$.  An analogous argument shows that $\Ext^1_{\cA}(\Phi,W') = 0$. 
\end{proof}

The following Lemmas establish criteria for a group scheme in $\cA$ to be filtered by copies of $\Phi$, culminating in Proposition \ref{Phi1}.  We thank Brian Conrad for his insight on the last part of the next proof.

\begin{lem} \label{NoMu}
Let $V \ne 0$ be in $\cA$, with no subgroup scheme isomorphic to $\Mu_2$.  Then $V$ is filtered entirely by copies of $\cZ_2$ or $V$ has a subgroup scheme isomorphic to $\Phi$.  If $W$ is a closed subgroup scheme of $V$ filtered entirely by copies of $\Phi$, then $V/W$ has no closed subgroup scheme isomorphic to $\Mu_2$.
\end{lem}

\begin{proof}
Assume that $V$ is not filtered entirely by copies of $\cZ_2$.  Then $V$ has a subquotient isomorphic to $\Mu_2$ by Proposition \ref{ZMuFil}.  Among all filtrations of $V$ whose successive quotients have order 2, choose $\cF$ such that the first copy of $\Mu_2$ occurs as early as possible: 
$$
\gr(\cF) = [\, \overbrace{\cZ_2, \dots, \cZ_2}^n \, , \Mu_2, \dots], 
$$
with $n$ minimal.  By hypothesis, $n \ge 1$.  Then the first subquotient graded by $[\cZ_2,\Mu_2]$ above does not split and therefore is isomorphic to $\Phi$:    
$$
\gr(\cF) = [ \, \overbrace{\cZ_2, \dots, \cZ_2}^{n-1} \, , \Phi, \dots],
$$
Let $W'$ be the closed subgroup scheme of $V$ filtered by the $n-1$ copies of $\cZ_2$ indicated above.  By Lemma \ref{bigExt}, $\Ext^1_{\cA}(\Phi,W')$ vanishes.  The resulting splitting produces a subgroup scheme of $V$ isomorphic to $\Phi$.

Next, suppose that there is a closed immersion $i\!: \, \Mu_2 \to V/W$.  Under the corresponding natural map $\Ext^1(V/W,W) \to \Ext^1(\Mu_2,W)$, the image of the class $[V]$ is represented by the pullback \cite[II.6]{HS}:
\begin{equation*} 
\begin{tikzcd}
P \ar["\beta" ' , d,dashed]  \ar["\alpha",r,dashed]  & \Mu_2 \ar["i" , d] \\
V \ar["j",r]  & V/W
\end{tikzcd}
\end{equation*}
Forming the pullback $P$ preserves monomorphisms, so $\beta$ also is a closed immersion.  The exact sequence $0 \to W \to P \to \Mu_2 \to 0$ representing the extension class $[P]$ splits by Lemma \ref{bigExt}, giving a closed subgroup scheme $U \simeq \Mu_2$ in $P$.  Since $\beta$ is a closed immersion, $\beta(U)$ is a copy of $\Mu_2$ in $V$, contradicting the hypothesis on $V$.
\end{proof}

\begin{lem} \label{Phi0}
Every group scheme $V$ in $\cA$ has a filtration with grading $[\cM,\cG,\cZ]$, where $\cM$ is filtered by copies of $\Mu_2$, $\cG$ by copies of $\Phi$ and $\cZ$ by copies of $\cZ_2$.
\end{lem}  

\begin{proof}  
Among closed subgroup schemes of $V$, choose $\cM$ of maximal dimension $\ge 0$ with $\cM$ is filtered entirely by copies of $\Mu_2$.  Then $W = V/\cM$ has no subgroup scheme isomorphic to $\Mu_2$ by maximality of $\cM$.  If $W = 0$ or $W$ is filtered entirely by copies of $\cZ_2$, we are done.  By Lemma \ref{NoMu}, we therefore assume that there is a closed subgroup scheme of $W$ isomorphic to $\Phi$.

Among closed subgroup schemes of $W$, choose $\cG$ of maximal dimension $\ge 2$ such that $\cG$ is filtered entirely by copies of $\Phi$.   Then $X = W/\cG$ has no closed subgroup scheme isomorphic to $\Phi$ by maximality of $\cG$.  Also, $X$ has no closed subgroup scheme isomorphic to $\Mu_2$ by Lemma \ref{NoMu}.   Then Lemma \ref{NoMu} tells us that $X = 0$ or $X$ is filtered entirely by copies of $\cZ_2$.  Thus, there is a filtration
$$
0 \subseteq \cM \subset V_2 \subseteq V
$$
with $\cM$ filtered entirely by copies of $\Mu_2$, $\cV_2/\cM \simeq \cG$ filtered entirely by copies of $\Phi$ and $V/V_2 \simeq X$, filtered entirely by copies of $\cZ_2$.  
\end{proof}

\begin{prop} \label{Phi1}
Let $V$ be a group scheme in $\cA$.  The following are equivalent:
\begin{enumerate}[\rm i)]
\item $V$ has a filtration $0 \subset V_1 \subset \dots \subset V_n = V$ with successive quotients isomorphic to $\Phi$.  \vspace{2 pt}
\item $V$ has no closed subgroup scheme isomorphic to $\Mu_2$ over $\Z[\frac{1}{p}]$ and $|V^m| = |V^{et}|$ as group schemes over $\Z_2$.
\item $V$ has no quotient isomorphic to $\cZ_2$ over $\Z[\frac{1}{p}]$ and $|V^m| = |V^{et}|$ as group schemes over $\Z_2$.
\end{enumerate}
If $W$ is a closed subgroup scheme of $V$ and both $V$ and $W$ satisfy these conditions, then so does $V/W$.
\end{prop}

\begin{proof}
\noindent We use induction on $n$ to show that (i) implies (ii).  The case $n=1$ follows from the definition of $\Phi$.  Assume $n \ge 2$, but that $V$ has a subgroup scheme $\cU$ isomorphic to $\Mu_2$ over $\Z[\frac{1}{p}]$.  The filtration in (i) provides an exact sequence
$$
0 \to V_{n-1} \to V \xrightarrow{\phi} \Phi \to 0.                                                                   
$$
in which the closed subgroup scheme $V_{n-1}$ of $V$ has a filtration by copies of $\Phi$.  Then, by the induction hypothesis, (ii) holds for $V_{n-1}$ and so $\cU$ is not contained in $V_{n-1}$.  Thus, $\phi(\cU) \simeq \Mu_2$ by connectedness.  But this contradicts the fact that $\Phi$ is a non-split extension of $\Mu_2$ by $\cZ_2$.  Hence $V$ has no subgroup scheme isomorphic to $\Mu_2$.  Also, condition (i) and the grading $[\cZ_2, \Mu_2]$ of a filtration for $\Phi$ imply that $|V^m| = |V^{et}| = n$, so (ii) holds for $V$.  

For the converse, assume (ii).  Then $\cM = 0$ in Lemma \ref{Phi0}, so there is a filtration of $V$ with grading $[\cG,\cZ]$.  Since $\vv{\cG^m} = \vv{\cG^{et}}$ on restriction of the base to $\Z_2$ and $|V^{et}| = |V^{m}|$ by assumption, we find that $\cZ$ is trivial.  Hence $V$ has a filtration with successive quotients isomorphic to $\Phi$.  Equivalence of item (iii) follows by duality.

For the last claim in the Proposition, assume $V$ has no subgroup scheme isomorphic to $\Mu_2$.  The same holds for $V/W$ by Lemma \ref{NoMu}.  In addition,
$$
|(V/W)^{et}| = |V^{et}|/|W^{et}| = |V^m|/|W^m| = |(V/W)^m|, 
$$
so $V/W$ satisfies (ii).
\end{proof}

\begin{Def} \label{balanced}
A group scheme $V$ in $\cA$ is {\em balanced} if it enjoys the equivalent properties in Proposition {\rm \ref{Phi1}}. 
\end{Def}

The rest of this section treats finite flat group schemes over $\Z_2$.   We first describe one such  group scheme $\Gamma$, used in a patching argument in the next section. According to \cite{LMF}, the Jacobian $B = J(\cC)$ of the curve   
$$
\cC\!:  \, y^2 + (x^3 + x^2 + x + 2) \,  y = (x^2 + 1)^2
$$
is isogenous to the modular Jacobian $J_0(29)$ with real multiplication by $\Z[\sqrt{2}]$.  It is principally polarized and has good reduction over $\Z_2$. Thus, $\Gamma = B[2]$ is a self-dual group scheme over $\Z_2$.

 Let $\cE$ be the non-trivial \'{e}tale extension of $\cZ_2$ by $\cZ_2$ over $\Z_2$ with $2 \cE = 0$  and denote its Cartier dual by $\cM = \cE^D$.   Thus $\cE$ and $\cM$ split over the unramified quadratic ring $\Z_2[\Mu_3]$.  Set $F = \Q_2(\Mu_{12})$ and 
\begin{equation} \label{LocalGal}
 \rG = \Gal(F/\Q_2) = \lr{\sigma, \phi},
\end{equation}
 where $\sigma(\zeta_{12}) = \zeta_{12}^7$ and $\phi(\zeta_{12}) = \zeta_{12}^5$.  Note that $\sigma$ generates the inertia subgroup in $\rG$ and $\phi$ is a choice of Frobenius fixing $i$.

\begin{lem}  \label{J29}
The group scheme $\Gamma$ is an extension of $\cE$ by $\cM$ over $\Z_2$ with field of points $\Q_2(\Gamma) = F$ and a unique $\Q_2$-rational point of order $2$.
\end{lem}

\begin{proof}
The field $\Q_2(\Gamma)$ is the splitting field of 
\begin{equation*} 
 f(x) = (x^3 + x^2 + x + 2)^2 + 4(x^2 + 1)^2  =  g(x) \, \ov{g}(x) 
\end{equation*}
where $g(x)$ and $\ov{g}(x)$ are conjugate cubics over $\Z_2[i]$. The congruence
$$
g(x) \equiv x(x^2+x+1) \pmod{2\Z[i]}
$$
implies that $\Q_2(\Gamma) = F$ by Hensel's Lemma.  Also, $f$ factors as an irreducible quadratic times an irreducible quartic over $\Z_2$.  Hence $\Gamma$ has exactly one point of order 2 over $\Q_2$.  

There is no biconnected subquotient of $\Gamma$ because $\Gal(F/\Q_2)$ is a 2-group.  The multiplicative subgroup scheme $\Gamma^m$ of $\Gamma$ and the \'{e}tale quotient $\Gamma^{et}$ have order 4 because $\Gamma$ is self-dual.  In the connected-\'{e}tale exact sequence
\begin{equation} \label{GammaExt}
0 \to \Gamma^m \to \Gamma \to \Gamma^{et} \to 0,
\end{equation}
$\sigma$ acts trivially on the Galois modules $\Gamma^m(F)$ and $\Gamma^{et}(F)$ by Lemma \ref{InertiaOnVm}.  By assumption, $\vv{\Gamma(\Q_2)} = 2$, so $\Gamma^m$ is a non-trivial extension of $\Mu_2$ by $\Mu_2$, isomorphic to $\cM$.  By duality, $\Gamma^{et}$ is isomorphic to $\cE$. 
\end{proof}

\begin{prop}  \label{GammaOverZ2}
The group scheme $\Gamma$ represents the unique class $[V]$ in $\Ext^1_{\Z_2}(\cE,\cM)$ such that $2V = 0$, $\Q_2(V) \subseteq F$ and $\vv{V(\Q_2)} = 2$.  

\end{prop}

\begin{proof}
We first show that the conditions on $V$ determine the associated Galois module $V(F)$.  The exact sequence 
$$
0 \to \cM  \xrightarrow{i} V \xrightarrow{j} \cE \to 0
$$ satisfied by $V$ leads to a filtration 
$
0 \subset V_1 \subset V_2 \subset V_3 \subset V
$ 
with grading $[\Mu_2,\Mu_2,\cZ_2,\cZ_2]$.  Choose $x_4$ in $V(F)$ but not in $V_3(F)$ and define additional elements of $V(F)$ by the following action of $\Gal(F/\Q) = \lr{\sigma,\phi}$:
\begin{equation} \label{GalOnGamma}
\begin{tikzcd}
x_4 \ar["\sigma-1" ' , d]  \ar["\phi-1",r]  &x_3 \ar["\sigma-1" , d] \\
x_2 \ar["\phi-1",r]  & x_1
\end{tikzcd}
\end{equation}
Since $\phi$ preserves $\cM(F)$ and $\cE(F)$, but $\phi$ is not trivial on $\cE(F)$, the cosets 
$$
\ov{x}_4 = j(x_4) = x_4 +\cM(F) \quad \text{ and } \quad \ov{x}_3 = j(x_3) = x_3 +\cM(F)
$$
span $\cE(F) = V(F)/\cM(F)$.  Also, $x_1$ and $x_2$ are in $\cM(F) = \ker j$ by Lemma \ref{InertiaOnVm}.  Since $\phi$ is an involution, $x_3$ is fixed by $\phi$.  Then $x_1 \ne 0$, or else $x_3$ is fixed by $\sigma$, so $\vv{V(\Q_2)}$ contains $\lr{V_1,x_3}$, of order 4, a contradiction.

In fact, $x_1, \dots, x_4$ is a basis for $V(F)$.  Otherwise, there is an $\F_2$-relation
\begin{equation} \label{DepRel}
\alpha_1 x_1 + \alpha_2 x_2 + \alpha_3 x_3 + \alpha_4 x_4 = 0.
\end{equation}
Apply $\phi-1$ and $\sigma-1$, to obtain:
\begin{equation} \label{AuxEqs}
{\rm i)} \hspace{5 pt} \alpha_2 x_1 + \alpha_4 x_3 = 0 \hspace{20 pt} \text{ and } \hspace{20 pt} {\rm ii)} \hspace{5 pt}\alpha_3 x_1 + \alpha_4 x_2 = 0.
\end{equation}
By (\ref{AuxEqs}i), $\alpha_4 = 0$, since $x_3$ is not in $\cM(F)$.  But then $\alpha_2 = 0$ in (\ref{AuxEqs}i), $\alpha_3 = 0$ in (\ref{AuxEqs}ii) and so $\alpha_1 = 0$ in \eqref{DepRel}.  Hence $\lr{x_1, \dots, x_4} = V(F)$ and $\lr{x_1,x_2} = \cM(F)$.

Let $V'(F)$ be another extension of Galois modules
$$
0 \to \cM(F) \xrightarrow{i'} V'(F)  \xrightarrow{j'} \cE(F) \to 0
$$
and choose $y_4$ in $V'(F)$ satisfying $j'(y_4) = \ov{x}_4$.  Let $f\!: \, V(F) \to V'(F)$ be the isomorphism determined by $f(x_4) = y_4$ and the action of Galois in \eqref{GalOnGamma} and let $y_t = f(x_t)$ for $t = 1, 2, 3$.  In particular, 
$$
j'(y_3) = j'((\phi-1)(y_4)) = (\phi-1)(j'(y_4)) = (\phi-1)(\ov{x}_4) = \ov{x}_3.
$$
Also, $i'(x_2)$ generates $\ker j'$ and so $i'(x_2) = y_2 + \beta y_1$ for some $\beta$ in $\F_2$.  Replace the choice of $y_4$ by $y_4' = y_4 + \beta y_3$ to arrange for 
$$
y_2' = (\sigma-1)(y_4') = y_2+\beta y_1 = i'(x_2)
$$
and redefine $f$ by $f(x_4) = y_4'$.  Then $f$ induces an isomorphism of extensions:
$$
\begin{CD}
0 @ >>> \cM(F) @> i >>     V(F) @>j>> \cE(F)  @>>> 0\\
  &&   @V \text{Identity} VV @VV f V            @VV  \text{Identity} V       \\
0 @ >>> \cM(F) @> i' >> V'(F) @> j' >> \cE(F)  @>>> 0\\
\end{CD}
$$

By Lemma \ref{J29}, $\Gamma$ satisfies the hypotheses on $V(F)$ and so $\Gamma(F)$ also represents the same extension class of Galois modules.  Then Lemma \ref{GalModGpSch} below implies that  $[V] = [\Gamma]$ in $\Ext^1_{\Z_2}(\cE,\cM)$.  
\end{proof}

\begin{lem}  \label{GalModGpSch}
Let $\cE$ and $\cM$ be arbitrary \'etale and multiplicative group schemes over $\Z_2$, respectively.  If the classes $[W_1]$ and $[W_2]$ in $\Ext^1_{\Z_2}(\cE,\cM)$ become equal in $\Ext^1_{\Q_2}(\cE,\cM)$, then $[W_1] = [W_2]$ in $\Ext^1_{\Z_2}(\cE,\cM)$.
\end{lem}

\begin{proof}
Let $[X] = [W_1]-[W_2]$, let $F' = \Q_2(X)$ and consider an exact sequence
$$
0 \to \cM \xrightarrow{f_1} X \xrightarrow{f_2} \cE \to 0
$$
representing $[X]$.  Since $[X]$ becomes trivial in $\Ext^1_{\Q_2}(\cE,\cM)$, there is a Galois submodule $Y$ of $X(F')$ such that $f_2$ induces an isomorphism $Y \xrightarrow{\sim} \cE(F')$.  Let $\cY$ be the closed subgroup scheme of $X$ whose associated Galois module is $Y$.  Observe that $\cY$ is \'{e}tale.  Otherwise, it has a non-trivial multiplicative component that would be preserved by the morphism $f_2$, to give a non-trivial multiplicative component of $\cE$.  Hence $\cY$ is \'etale.  But then $\cY$ and $\cE$ can be identified with their associated Galois modules, which are isomorphic by construction.  Therefore $[X] = 0$ in $\Ext^1_{\Z_2}(\cE,\cM)$ and so $[W_1] = [W_2]$. 
\end{proof}

The last result presented here will be used  in the next section, to ascertain uniqueness of the group scheme $\Xi_p$.

\begin{lem} \label{EndoGamma} 
The natural map $\End_{\Z_2} \Gamma \to \End_{\Q_2} \Gamma$ is an isomorphism.
\end{lem}

\begin{proof}
By Kummer theory, there is a unique extension of $\cZ_2$ by $\Mu_2$ of exponent 2 with field of points $\Q_2(i)$, namely the Katz-Mazur group scheme $G_{-1}$ over $\Z_2$.  Let $x_1, \dots, x_4$ be the $\F_2$-basis for $\Gamma(F)$ in \eqref{GalOnGamma}.  It follows that the closed subgroup scheme $W$ of $\Gamma$ corresponding to the Galois module $\lr{x_1,x_3}$ is isomorphic to $G_{-1}$.  The Galois module associated to the quotient $\Gamma/W$ is given by $\lr{x_2+W,x_4+W}$ and so $\Gamma/W$ also is isomorphic to $G_{-1}$.  Thus $\Gamma$ is an extension of $G_{-1}$ by $G_{-1}$ which splits upon base change to the unramified ring $S = \Z_2[\Mu_3]$.   

Let $\tilde{\Gamma} = \Gamma \otimes_{\Z_2} S$ and $\tilde{G}_{-1}= G_{-1} \otimes_{\Z_2}  S$.  Then $\tilde{\Gamma} \simeq \tilde{G}_{-1} \oplus \tilde{G}_{-1}$ and 
$\End_S \tilde{\Gamma}$ is the ring of $2 \times 2$ matrices $\rM_2(\End_S \tilde{G}_{-1})$.
Taking the action of $\sigma$ into account, we find that 
$
\End_{\Z_2} G_{-1}  = \End_{\Q_2} G_{-1} = \F_2^2.
$
Indeed, we have the identity morphism and the morphism
$$
G_{-1} \to G_{-1}/\Mu_2 \simeq \cZ_2 \xrightarrow{h} \Mu_2 \to  G_{-1},
$$
where $h$ generates $\Hom_{\Z_2}(\cZ_2,\Mu_2)$.   The induced maps generate all 4 endomorphisms of the Galois module $G_{-1}(\Q_2)$.  By the same argument over $S$, we have $\End_S(\tilde{G}_{-1}) = \End_{\Q_2(\Mu_3)}(\tilde{G}_{-1}) = \F_2^2$. 

Given $f$ in $\End_{\Q_2} \Gamma $, we can therefore find some $g$ in $\End_S \tilde{\Gamma}$ such that $f$ and $g$ agree upon base change to $\Q_2(\Mu_3)$. Since $f$ is invariant under the action of $\Gal(\Q_2(\Mu_3)/\Q_2)$, so is $g$.  Hence $g$ is in $\End_{\Z_2}(\Gamma)$, as required.  
\end{proof}

\section{The group scheme $\Xi_p$ and $(d,\Phi$)-blocs} \label{dPhiSection}

To construct these group schemes, we first review a particular representation of the dihedral group $\rD_{2^{e+1}}$ of order $2^{e+2}$.   Let $g_1,g_2$ be non-commuting involutions that generate $\rD_{2^{e+1}}$ and let $\cR = \F_2[\rD_{2^{e+1}}]$ be its group ring.  The augmentation ideal $\cI = (g_1-1, g_2-1)$ is the unique maximal ideal of $\cR$ and is nilpotent.  Define $\alpha_j$ and $\beta_j$ in $\cI^j$ by
$$
\begin{array}{l l l}
\alpha_j &=& \cdots (g_1-1)(g_2-1)(g_1-1), \\[2 pt] 
\beta_j &=& \cdots (g_2-1)(g_1-1)(g_2-1), \\ 
\end{array}
$$
each having $j \ge 1$ factors.  From the relations $(g_1-1)^2 = 0$, $(g_2-1)^2 = 0$ and the basic identity  
$$
(g g'-1) = (g-1)(g'-1) + (g-1) + (g'-1) \, \text{ for all } g, g' \text{ in }\rD_{2^{e+1}},
$$
we find that 
\begin{equation} \label{SpanR}
\cR = \lr{1, \alpha_1, \beta_1, \alpha_2, \beta_2, \dots} \ \text{over } \F_2. 
\end{equation}

\begin{prop} \label{GalRep}
For $e \ge 1$, let $\rD_{2^{e+1}} = \lr{g_1,g_2}$ be the dihedral group of order $2^{e+2}$ generated by involutions $g_1,g_2$.  Let $\cR = \F_2[\rD_{2^{e+1}}]$ and let $\rX = \cR x$ be a cyclic $\cR$-module with $g_2(x) = x \ne 0$.  Let $k$ be the minimal integer such that $\alpha_k \, x = 0$ and assume that $k \ge 3$.  Then $\dim_{\F_2} \rX = k$ with $k \le 2^{e+1}$ and $\rX$ is indecomposable.  The representation 
$
\rho_{\rX}:  \, \rD_{2^{e+1}} \to \SL_k (\F_2)
$
afforded by $X$ is faithful if and only if $2^e <  k \le 2^{e+1}$, or equivalently, $e+1 = \lceil \log_2 k \rceil$.   
\end{prop}

\begin{proof}
There is a minimal $k$ such that $\alpha_k \, x = 0$ because $\cI$ is nilpotent.  Then we have the chain
\begin{equation} \label{basis}
x = v_k \xrightarrow{g_1-1} v_{k-1} \xrightarrow{g_2-1} v_{k-2}  \xrightarrow{\quad} \dots  \xrightarrow{\quad} v_2 \xrightarrow{\quad} v_1  \xrightarrow{g_\delta -1} 0, 
\end{equation}
where $\delta = 1$ or 2, in agreement with the parity of $k$.  Thus, $g_1$ and $g_2$ act by:
\begin{equation} \label{basis2}
\begin{array} {r l c l l}
 & (g_1-1)(v_j)  = v_{j-1} &{\rm and}& (g_2-1)(v_j) = 0 &{\rm if \,}  j \equiv k \bmod{2}, \\[3 pt]
 & (g_1-1)(v_j)  = 0 &{\rm and}& (g_2-1)(v_j) = v_{j-1} &{\rm if \,}  j \not\equiv k \bmod{2}
\end{array}
\end{equation}
By \eqref{SpanR}, $v_1, \dots, v_k$ span $\rX$.  Also $v_1 = \alpha_{k-1}x \ne 0$ by assumption.  

To verify linear independence of $v_1, \dots, v_k$ assume instead that there is a non-trivial dependence relation over $\F_2$.  Then there is one of the form 
\begin{equation} \label{ind}
c_1 v_1 + c_2 v_2 +  \dots + c_{m-1} v_{m-1}+ v_m = 0
\end{equation}
with $m$ minimal and necessarily, $m \ge 2$.  Depending on the parity of $m$ and $k$, there is some $g_\epsilon$, with $\epsilon = 1$ or 2, such that $(g_\epsilon-1)(v_m) = v_{m-1}$.  Applying $g_{\epsilon}-1$ to \eqref{ind} gives a shorter non-trivial relation, so a contradiction.  Hence $v_1, \dots, v_k$ is a basis for $\rX$ and $\dim \rX = k$.  

Let $v_j = 0$ if $j \le 0$ and let $t = \rho_{\rX}(g_2g_1)$.  If $1 \le j \le k$, we have
\begin{equation} \label{t-1}
(t-1)(v_j) = \begin{cases} v_{j-1}&\text{if } j \not\equiv k \bmod{2}, \\ v_{j-1} + v_{j-2} &\text{if }  j \equiv k \bmod{2}, \end{cases}
\end{equation}
so $(t-1)^{k-1}(v_{k}) = v_1 \ne 0$.  But $(t-1)^{k}(\rX) = 0$ because $\dim \rX = k$ and $t$ is unipotent.  Hence the index of nilpotency of $t-1$ is exactly $k$.  Thus $t$ can be represented by one $k \times k$ Jordan block and so $\rX$ is indecomposable.

It also follows that $t$ has order $2^a$ where $a = \lceil \log_2 k \rceil$.  Since the maximal order among elements of $\rD_{2^{e+1}}$ is $2^{e+1}$, we find that $k \le 2^a \le 2^{e+1}$. Furthermore, the representation $\rho_{\rX}$ is faithful precisely when $t$ has order $2^{e+1}$, or equivalently, $2^e < k \le 2^{e+1}$.  
\end{proof}

\begin{Rem} \label{RingConverse}
The following cases were excluded above.  
\begin{enumerate}[\, $\cdot$]
\item $k=1 \Rightarrow \lr{g_1,g_2}$ acts trivially on $X =\lr{v_1}$.  \vspace{2 pt}
\item $k = 2 \Rightarrow \rho_X{\lr{g_1,g_2}} \simeq \Z/2\Z$ on $X = \lr{v_1,v_2}$. 
\end{enumerate} 
Note a converse to the Proposition.  Let $V$ be an $\F_2$-vector space of dimension $k  \ge 3$.  If $\lr{g_1,g_2}$ acts faithfully on a basis for $V$ as in \eqref{basis2}, then $g_1$ and $g_2$ are non-commuting involutions and $G \simeq \cD_{2^{e+1}}$, with $e+1 = \lceil \log_2 k \rceil$.
\end{Rem}

\begin{Not} \label{Ringel}
Let $\rX_k$ denote the module of $\F_2$-dimension $k$ in Proposition \ref{GalRep}.  We call $v_1, \dots, v_k$, a {\em Ringel basis}\footnote{Ringel \cite{Rin} determined all the indecomposable representations of $\rD_{2^{e+1}}$, among which is the cyclic module $\rX_k$.  For the reader's convenience, we included a self-contained proof for our special case.} for $\rX_k$ when the action of $g_1$ and $g_2$ is given by \eqref{basis2}.  By convention, $v_j = 0$ if $j$ is not in $\{1, \dots, k\}$.
\end{Not}

\begin{lem}  \label{tAction}
Set $t = \rho_{\rX_k}(g_2 g_1)$ and suppose that $1 \le 2^a \le k$. The action of $t^{2^a}$ on $v_j$ is determined by {\rm \eqref{t-1}} and the relation
\begin{equation} \label{LinAlg}
t^{2^a}(v_j) = v_j + t^{2^{a-1}}(v_{j-2^a}) \quad \text{for } a \ge 1.
\end{equation}
In particular, if $2^{a+1} \le k$, then 
\begin{equation} \label{tPower}
t^{2^a}(v_{2^{a+1}}) = v_1+v_2+v_4+v_8+ \dots + v_{2^{a+1}}.
\end{equation}
\end{lem}

\begin{proof}
If $u = g_1 g_2 + g_2 g_1 =  t+t^{-1}$ and $j \le k$, then $u(v_j) =  v_{j-2}$ and by induction, $u^{2^{a-1}}(v_j) = v_{j-2^a}$, so 
$$
\begin{array}{r c l}
(t-1)^{2^a}(v_j) &=& (t-1)^{2 \cdot 2^{a-1}}(v_j)  = (t^2-1)^{2^{a-1}}(v_j) = (tu)^{2^{a-1}}(v_j)  \\[4 pt] 
&=& t^{2^{a-1}} u^{2^{a-1}}(v_j) 
= t^{2^{a-1}}(v_{j-2^a}). 
\end{array}
$$ 
Thus, \eqref{LinAlg} holds.  Since $t(v_2) = v_2 + v_1$, \eqref{tPower} can then be verified by induction.
\end{proof}

\vspace{5 pt}

See Proposition \ref{RayClass1a} for notation used throughout the rest of this section.  In particular, $M$ is the maximal 2-primary abelian extension of $K = \Q(\sqrt{-p})$ of ray class conductor $2\cO_K$ and $M_u$ is the maximal subfield of $M$ unramified over $K$.  We fix a prime $\lambda = \lambda_M$ over 2 in $M$ and $\lambda_u$ below $\lambda$ in $M_u$.  We also fix a prime $v$ over $p$ in $M$.  Generators for $\Gal(M/\Q) = \lr{\sigma_\lambda, \sigma_v}$ are specified in Corollary \ref{RayClass1b}.  

Let  $\rX_{h_2}$ be the faithful representation of $\Gal(M_u/\Q) \simeq \rD_{h_2}$ in Notation \ref{Ringel}, taking $g_1 = \sigma_{\lambda}$ and $g_2 = \sigma_v$.  Let $v_1, \dots, v_{h_2}$ be the resulting Ringel basis for $\rX_{h_2}$.  See Lemma \ref{J29}ff for the group scheme $\Gamma$ over $\Z_2$ with field of points $F = \Q_2(\Mu_{12})$.  

\begin{theo} \label{GrpSch}
There is a group scheme  $\Xi_p$ over $R = \Z[\frac{1}{p}]$ with field of points $\Q(\Xi_p) = M_u$, characterized by the following isomorphisms: \vspace{ 2 pt}

\centerline{ {\rm i)} the Galois module $\Xi_p(M_u) \simeq \rX_{h_2}$  \quad and \quad {\rm ii)} \, $\Xi_p \otimes_R {\Z_2} \simeq \oplus_1^{h_2/4} \, \Gamma$.}
\end{theo}
 
\begin{proof}
The Galois module $\rX_{h_2}$ corresponds to an \'{e}tale group scheme $\cX$ over $R' = \Z[\frac{1}{2p}]$ with field of points $M_u$ \cite[\S3.6]{Ta2}.  Set $h_2 = 2^{n+1}$.  We will give an isomorphism of the \'etale group schemes 
$$
\cX \otimes_{R'} \Q_2 \quad \text{and} \quad (\oplus_1^{2^{n-1}} \, \Gamma) \otimes_{\Z_2} \Q_2
$$ 
over $\Q_2$, determined by the choice of a Ringel basis $\{v_1, \dots, v_{h_2}\}$ for $\rX_{h_2}$.  Then the patching argument in Lemma \ref{Patch} implies that $\cX$ prolongs to a group scheme $\Xi_p$ over $R = \Z[\frac{1}{p}]$.

 By Corollary \ref{RayClass1b}, the decomposition group $\cD_{\lambda_u}(M_u/\Q)$ is a Klein 4-group, generated by $\sigma_{\lambda_u} = \sigma_\lambda \vert M_u$ and a Frobenius $\phi_{\lambda_u} = (\sigma_{\lambda_u} \sigma_v)^{2^n}$.  Let $\cR_u = \F_2[\cD_{\lambda_u}(M_u/\Q)]$.  
 
First we show that for $j$ in $J = \{ j \text{ even} \mid 2^n+2 \le j \le 2^{n+1}\}$, the cyclic $\cR_u$-submodule of $\rX_{h_2}$ given by $N_j = \cR_u v_j$ is isomorphic to $\Gamma(F)$.  The nilpotent operators $a = \sigma_{\lambda_u} -1$ and $b = \phi_{\lambda_u}-1$ commute and $a^2 = b^2 = 0$.  By \eqref{basis2}, we have
\begin{equation} \label{LocalGalois}
v_{j-1} = a(v_j) \quad \text{and} \quad N_j = \lr{v_j,\, v_{j-1}, \, b(v_j), \, b(v_{j-1})}.
\end{equation}
This action of $\cD_{\lambda_u}(M_u/\Q)$ agrees with the action of the local Galois group on $\Gamma(F)$ in \eqref{GalOnGamma}, with $F = \Q_2(\Mu_{12})$.

Consider the increasing filtration of $\rX_{h_2}$ given by 
\begin{equation*}
\rX_{h_2}^{(k)} = \lr{v_1, v_2, \dots, v_k} \, \text{ for } \,  k = 1, \dots, h_2.
\end{equation*}
and trivial for $k \le 0$.  By \eqref{LinAlg}, we have the congruences:
\begin{equation} \label{Nj}
\begin{array}{r c l l}
b(v_j) &\equiv& v_{j-2^n} &\text{mod } \rX_{h_2}^{(j-2^n -2^{n-1})} \quad \text{ and }  \\[3 pt]
b(v_{j-1}) &\equiv& v_{j-1-2^n}   &\text{mod } \rX_{h_2}^{(j-1-2^n -2^{n-1})}
\end{array}
\end{equation}
It follows that the $\F_2$-generators of $N_j$ in \eqref{LocalGalois} are linearly independent, $\dim_{\F_2} N_j = 4$ and $N_j \simeq \Gamma(F)$ as $\cR_u$-modules.   

By \eqref{Nj} and induction, $N = \sum_{j \in J} N_j$ is a direct sum of the $N_j$'s as an $\cR_u$-submodule of $\rX_{h_2}$.  Since $\vv{J} = 2^{n-1}$, we find that 
$$
\dim_{\F_2} N = 2^{n-1} \vv{\Gamma} = 2^{n+1} = h_2  = \dim_{\F_2} \rX_{h_2} 
$$ 
and so $\rX_{h_2}$ and $\oplus_1^{2^{n-1}} \, \Gamma(F)$ are isomorphic as $\cR_u$-modules.  Equivalently, $\cX \otimes_{R'} \Q_2$ is isomorphic to $(\oplus_1^{2^{n-1}} \, \Gamma) \otimes_{\Z_2} \Q_2$.  This concludes the construction of $\Xi_p$ by patching.  

Let $\Theta = \oplus_1^{2^{n-1}} \, \Gamma$.  Since $\End_{\Z_2} \Theta = \rM_{2^{n-1}}(\End_{\Z_2} \Gamma)$,  surjectivity of the homomorphism $\Hom_{\Z_2}(\Theta,\Theta) \to \Hom_{\Q_2}(\Theta,\Theta)$ follows from Proposition \ref{GammaOverZ2}.  Furthermore, items (i) and (ii) and the identification of Galois modules over $\Q_2$ constructed above characterize $\Xi_p$, by Lemma \ref{Patch}.
\end{proof}

\vspace{5 pt}

Although $\Q(\Xi_p) = M_u$, it is convenient to take points in the field $M$ and use the fixed prime $\lambda$ over 2 in $M$, to simplify notation.  We write $\Xi_p^m$ for the multiplicative subgroup scheme of $\Xi_p$ over $\Z_2$.  Then $\Xi_p^m(M,\lambda)$ denotes the corresponding  $\cD_{\lambda}(M/\Q)$-module and $\Xi_p^{et}(M,\lambda) = \Xi_p(M)/\Xi_p^m(M,\lambda)$.

\begin{cor} \label{MuZXi}
We have $\F_2$-bases:  \, $\Xi_p^m(M,\lambda) = \lr{v_j \mid j \, {\rm odd
}}$ and 
$$
\Xi_p^{et}(M,\lambda) = \lr{v_j  + \Xi_p^m(M,\lambda) \mid j \, {\rm even}}
$$ 
with $1 \le j \le h_2$.  Let $\Xi^{(j)}$ be the closed subgroup scheme of $\Xi_p$ attached to the Galois submodule $\lr{v_1, \dots, v_j}$ of $\Xi_p(M)$.   Successive quotients of 
$$
0 = \Xi_p^{(0)} \subset \Xi^{(2)}_p \subset \Xi^{(4)}_p \dots \subset \Xi^{(h_2)}_p
$$
are isomorphic to the Katz-Mazur group scheme $G_{-1}$.
\end{cor}

\begin{proof}
Since $\lr{v_j \mid j \text{ odd}} = (\sigma_{\lambda} - 1)(\Xi_p(M))$, it is contained in $\Xi_p^m(M,\lambda)$ by Remark \ref{InertiaOnVm}.  Self-duality of $\Gamma$ over $\Z_2$ implies that 
$$
\dim  \Xi_p^m(M,\lambda) = \dim \Xi_p^{et}(M,\lambda) = \textstyle{\frac{1}{2} } \dim \Xi_p(M) = \textstyle{\frac{1}{2} } h_2.
$$
Thus, $\Xi_p^m(M,\lambda) = \lr{v_j \mid j \, {\rm odd} }$ and $\Xi_p^{et}(M,\lambda) = \lr{v_j  + \Xi_p^m(M,\lambda) \mid j \, {\rm even} }$.   Each quotient $\Xi_p^{2j+2}/\Xi_p^{2j}$ is an extension of $\cZ_2$ by $\Mu_2$ of exponent 2 with field of points $\Q(i)$.  By Kummer theory, $G_{-1}$ represents this extension class.
\end{proof}

Next, we define certain subquotients of $\Xi_p$ for our analysis of $A[2]$ when $A$ is a prosaic abelian variety. 

\begin{Not} \label{dPhi}
Fix a Ringel basis $v_1, \dots, v_{h_2}$ for $\Xi_p(M)$.  Suppose that $2d+2 \le h_2$ and let $\Xi_p^{(2d+1)}$ be the closed subgroup scheme of $\Xi_p$ associated to the Galois module $\lr{v_1, \dots v_{2d+1}}$.  Set $\Phi_d = \Xi_p^{(2d+1)}\!/\Xi_p^{(1)}$.  In particular $\Xi_p^{(1)} \simeq \Mu_2$ and $\Phi_1 \simeq \Phi$. 
\end{Not}

The following definition will be used to highlight properties of $\Phi_d$.  By Proposition \ref{FGField}, if $V$ is in $\cA$ and $2V = 0$, then $\Q(V)$ is contained in $M$.  Fix the primes $\lambda = \lambda_M$ over 2 and $v$ over $p$ in $M$.  Recall that 
$$
\cR = \F_2[\Gal(M/\Q)] = \lr{\sigma_\lambda, \sigma_v}.
$$

\begin{Def}  \label{PhiBloc}
A group scheme $V$ in $\cA$ with $2V = 0$ is a $\Phi$-bloc if it is cyclic as a Galois module and has a filtration with exactly $d \ge 1$ successive quotients isomorphic to $\Phi$.  To specify $d$, we say $V$ is a $(d,\Phi)$-bloc.
\end{Def}

\begin{prop} \label{BlocProp}
Let $V$ be a group scheme of exponent $2$ in $\cA$.  Then $V$ is a $\Phi$-bloc if and only if both of the following conditions hold:
\begin{enumerate}[{\rm i)}]
\item There is a choice of $x$ in $V^m(M,\lambda)$ generating $V(M) = \cR x$. 
\vspace{2 pt} 

\item $V$ has no closed subgroup scheme isomorphic to $\Mu_2$. 
\end{enumerate} 
If so, $\dim_{\F_2} V(M) = 2d$ is even and $V(M)$ has a Ringel basis such that 
\begin{equation} \label{BlocRingel}
x = w_{2d} \xrightarrow{\sigma_v-1} w_{2d-1} \xrightarrow{\sigma_{\lambda}-1} w_{2d-2} \to \dots
\end{equation}
with $V^m(M,\lambda) = \lr{w_j \mid j \, {\rm even}}$ and 
$
V^{et}(M,\lambda) = \lr{w_j + V^m(M,\lambda) \mid j \, {\rm odd}}.
$  
If $d \ge 2$, then $\Gal(\Q(V)/\Q)$ is the dihedral group $\rD_{2^{e+1}}$, where $e = \lceil \log_2 d \rceil$.
\end{prop}  

\begin{proof}
Assume that $V$ is a $\Phi$-bloc.  Let $\cI$ be the augmentation ideal in $\cR$ and let $\cI(V(M))$ be the image of $\cI$ acting on $V(M)$.  By  Lemma \ref{GalToGrpSch}, there is a closed subgroup scheme $W$ of $V$ such that $W(M) = \cI(V(M))$.  Since $V(M)$ is a cyclic $\cR$-module, the quotient $V(M)/W(M)$ is 1-dimensional over $\F_2$.  But $V$ has no quotient isomorphic to $\cZ_2$ by Proposition \ref{Phi1}(iii), so $V/W \simeq \Mu_2$.  Let $\psi$ be the natural map to the quotient in the exact sequence 
$$
0 \to W \to V \xrightarrow{\psi} \Mu_2 \to 0.
$$  
By restricting the base to $\Z_2$ and taking the $\cD_{\lambda}(M/\Q)$-modules associated to multiplicative components, we obtain the exact sequence
$$
0 \to W^m(M,\lambda) \to V^m(M,\lambda) \xrightarrow{\psi} \Mu_2 \to 0.
$$
Thus there is an element $x$ in $V^m(M,\lambda)$ such that $\Mu_2 = \lr{\psi(x)}$ and so (i) holds by Nakayama's Lemma.  Also, item (ii) holds by Proposition \ref{Phi1}(ii).

For the converse, assume (i) and (ii).  By Lemma \ref{InertiaOnVm}, $\sigma_{\lambda}(x) = x$, since $x$ is in $V^m(M,\lambda)$.  Form a Ringel basis $x=w_k, w_{k-1}, \dots, w_1$ for $V(M)$ as in \eqref{BlocRingel}.  If $k = \dim V(M)$, then $k = 2d$ is even.  Otherwise, $w_1 = (\sigma_{\lambda}-1)(w_2)$ is in $V^m(M,\lambda)$ by Lemma \ref{InertiaOnVm}.  But then the closed subgroup scheme of $V$ corresponding to $\lr{w_1}$ is isomorphic to $\Mu_2$, contradicting (ii).

If $d = 1$, then $V$ is a non-split extension of $\Mu_2$ by $\cZ_2$, so $V \simeq \Phi$ and \eqref{BlocRingel} holds with trivial action of $\sigma_\lambda$.   Assume that $d \ge 2$, let $e = \lceil \log_2 d \rceil$ and let $L = \Q(V) \subseteq M$.   By Proposition \ref{GalRep} and Remark \ref{RingConverse}, $\Gal(L/\Q) \simeq \rD_{2^{e+1}}$.

If $2 \le j \le 2d-2$ is even, then $w_j$ is in $(\sigma_{\lambda}-1)(V(M))$ and therefore in $V^m(M,\lambda)$.   Since $w_{2d} = x$ is in $V^m(M,\lambda)$ by construction, we have $\dim W^m(M,\lambda) \ge d$.  In addition, the cosets
$w_j + V^m(M,\lambda)$ with $j$ odd and $1 \le j \le 2d-1$ are linearly independent in $V^{et}(M,\lambda) = V(M)/V^m(M,\lambda)$.  Otherwise, there is some $\F_2$-linear combination
$$
\sum_{j=0}^{d-1} \, \alpha_{2j+1} w_{2j+1} \,  \in \,  V^{m}(M,\lambda).
$$
Apply $\sigma_{\lambda}-1$, to find that $\sum_{j=1}^{d-1} \alpha_{2j+1} w_{2j} = 0$.  By linear independence, $\alpha_{2i+1} = 0$ for $i \ge 1$ and so $\alpha_1 w_1$ is in $V^m(M,\lambda)$.  Since $w_1$ is not in $V^m(M,\lambda)$, we have $\alpha_1 = 0$ as well.   It follows that $V^m(M,\lambda) = \lr{w_j \mid j \text{ even}}$ and 
$$
V^{et}(M,\lambda) = \lr{w_j + V^m(M,\lambda)  \mid j \text{ odd}},
$$ 
each of $\F_2$-dimension $d$.  By Proposition \ref{Phi1}, $V$ is filtered by copies of $\Phi$ and so $V$ is a $\Phi$-bloc.
\end{proof}

\begin{theo} \label{dPhiThm}
A $(d,\Phi)$-bloc $V$ exists if and only if $2d+2 \le h_2$, with $\Phi_d$ as an example.  Suppose that $d \ge 2$, let $e = \lceil \log_2 d \rceil$ and let $L = \Q(V)$.  Then: 
\begin{enumerate}[{\rm i)}]
\item $\Gal(L/\Q)$ is isomorphic to the dihedral group $\rD_{2^{e+1}}$, $L$ is the cyclic extension of $K$ of degree $2^{e+1}$ contained in $M_u$ and $2^{e+1} \vert h_2$.
\item If $d$ is a power of $2$, then $L$ is contained in $M_s$ and $2^{e+2} \vert h_2$.
\item In general, $8 \vert h_2$ and $p$ is in {\bf P1}.
\end{enumerate}
\end{theo}

\begin{proof}
Assume that $2d+2 \le h_2$ and let $v_1, \dots, v_{2d+1}$ be a Ringel basis for $\Xi_p^{(2d+1)}(M_u)$.  Then a Ringel basis for $\Phi_d$ is given by the cosets
\begin{equation} \label{PhidBasis}
y_j = v_{j+1} + \Xi_p^{(1)}, \hspace{20 pt}  j = 1, \dots, 2d
\end{equation}
Thus $\Phi_d(M)$ is a cyclic $\cR$-module of exponent 2.  Corollary \ref{MuZXi}, implies that the successive quotients  $\Phi_j/\Phi_{j-1}$ in the filtration
$$
0 = \Phi_0 \subset \Phi_1 \subset \dots \subset \Phi_{d}
$$
are extensions of $\Mu_2$ by $\cZ_2$ with field of points $\Q(\sqrt{p})$, therefore isomorphic to $\Phi$.  Hence $\Phi_d$ satisfies Definition \ref{PhiBloc} for a $(d,\Phi)$-bloc.

Next we verify items (i), (ii), (iii) for any $(d,\Phi)$-bloc $V$ with $d \ge 2$, using the Ringel basis in \eqref{BlocRingel}.  Then $\Gal(L/\Q) \simeq \cD_{2^{e+1}}$ and $L$ is a cyclic extension of $K$ of degree $2^{e+1}$ contained in $M$ by Proposition \ref{BlocProp}.  Also, $w_1$ is not in $V^m(M,\lambda)$.  By Corollary \ref{RayClass1b}, the decomposition group $\cD_{\lambda}(M/K)$ is given by 
$$
\cD_{\lambda}(M/K) = \Gal(M/M_s) = \lr{\tau_{\lambda}},
$$
where $\tau_{\lambda} = (\sigma_{\lambda} \sigma_v)^{h_2/2}$ has order 4.   If $h_2 \le 2d$ then, by Lemma \ref{tAction}, 
$$
\tau_{\lambda}(w_{h_2}) = w_1 + w_2 + w_4 + \dots + w_{h_2}
$$
is in $w_1 + V^m(M,\lambda)$ and therefore, not in $V^m(M,\lambda)$.  But $\tau_{\lambda}$ preserves $V^m(M,\lambda)$.  This contradiction implies that 
$$
h_2 \ge 2d+1 > 2^e +1,  \, \text{ where } \, e = \lceil \log_2 d \rceil.
$$
Since $h_2$ is a power of 2, we find that $2^{e+1} \vert h_2$.  But $\fdeg{L}{K} = 2^{e+1}$ and $\fdeg{M_u}{K}= h_2$, so $L$ is contained in $M_u$ by Corollary \ref{RayClass1b}.  If $d$ is a power of 2, then $d = 2^e$ and so $2^{e+2} \vert h_2$.  Then $L$ is contained in $M_s$, since $\fdeg{M_s}{K} = h_2/2$.  In all cases, $2d+2 \le h_2$.  Also, $8 \vert h_2$ and so $p$ is in $\text{\bf P1}$ by \eqref{h2}.  
\end{proof}

\begin{lem}  \label{VaVb}
A $(d,\Phi)$-bloc $V$ has a filtration by $(j,\Phi)$-blocs $V_j$:
$$
0 \subset V_1 \subset V_2 \subset \dots \subset V_{d-1} \subset V.
$$
Furthermore, $W_{d-j} = V/V_j$ is a $(d-j,\Phi)$-bloc and $V$ is a non-trivial extension of $W_{d-j}$ by $V_j$.  In particular, $V_1$ is the unique closed subgroup scheme of $V$ isomorphic to $\Phi$ and $\Hom_{\cA}(\Phi,V) = \F_2$.
\end{lem}

\begin{proof}
Fix a Ringel basis $w_1, \dots, w_{2d}$ for $V$ as in \eqref{BlocRingel} and let $V_j$ tbe the closed subgroup scheme of $V$ corresponding to the cyclic Galois module $\lr{w_1, \dots w_{2j}} = \cR w_{2j}$.  Then $V_j$ is a $(j,\Phi)$-bloc.  The quotient $W_{j'} = V/V_j$ is filtered by $j' = d-j$ copies of $\Phi$ and is cyclic as a Galois module, so it is a $(j',\Phi)$-bloc.  Since the Galois module $V(M)$ is indecomposable, $V$ represents a non-trivial extension class of $W_{j'}$ by $V_j$. 

If $U$ is a closed subgroup scheme of $V$ isomorphic to $\Phi$, then its Galois module $U(M)$ is annihilated by the ideal 
\begin{equation*}
\ga = ((\sigma_v-1)(\sigma_{\lambda}-1), \, (\sigma_{\lambda}-1)(\sigma_v-1)).
\end{equation*}
Since $\lr{w_1,w_2}$ is the unique 2-dimensional Galois submodule of $V(M)$ with this annihilator ideal, $U$ is the corresponding closed subgroup scheme of $V$, namely $V_1$.  Let us identify $\Phi$ with $V_1$.  If $f$ is in $\Hom_{\cA}(V_1,V)$ then $f$ preserves the multiplicative component $V_1^m(M,\lambda) = \lr{w_2}$.  It follows that $\Hom_{\cA}(V_1,V) = \F_2$, generated by the closed immersion $V_1 \to V$.
\end{proof} 

\vspace{5 pt} 

Recall the definitions of {\bf P1} and ${\bf P1^*}$ in \eqref{P1P3} and \eqref{P1*}.  Schoof  \cite{Sch2a} showed that if $p \equiv 1 \mymod{8}$, then
\begin{equation} \label{Schoof2a}
\dim_{\F_2} \Ext^1_{\cA}(\Phi,\Phi)=\begin{cases} 2\,\text{ if }\, p \in {\bf P1^*}, \\ 1\,\text{ otherwise. }\end{cases}
\end{equation}
(Schoof also treats $p \equiv 7 \pmod{8}$, not required here.)  Although the above result does not appear in the published version, what we need can be gleaned from \cite[Prop.~2.5, 2.6]{Sch4}, as follows.

\begin{prop}[Schoof] \label{SchoofDim}
If $p \equiv 9 \mymod{16}$ is in {\bf P1}, then $\Ext^1_{\cA}(\Phi,\Phi)$ has order $2$, generated by $[\Phi_2]$.  Any $(2,\Phi)$-bloc $V_2$ is an extension of $\Phi$ by $\Phi$ such that  $[V_2] = [\Phi_2]$ and $V_2 \simeq \Phi_2$.
\end{prop}

\begin{proof}
Since $p$ is in {\bf P1}, $h_2$ is divisible by 8, so $\Phi_2$ can be constructed.  It represents a non-trivial class in $\Ext^1_{\cA,[2]}(\Phi,\Phi)$ because its Galois module is indecomposable.

Applied to $0 \to \cZ_2 \xrightarrow{i} \Phi \xrightarrow{j} \Mu_2 \to 0$, the long exact sequence for $\Hom_\cA(-,\Phi)$ gives exactness of
$$
\to \Ext_{\cA}^1(\Mu_2,\Phi) \to \Ext^1_{\cA}(\Phi,\Phi)  \xrightarrow{j^*} \Ext^1_{\cA}(\cZ_2,\Phi) \to
$$
It follows from \cite[Prop.\!~2.5]{Sch4} that
$$
\Ext^1_{\cA}(\Mu_2,\Phi) = 0 \quad \text{and} \quad \dim_{\F_2}\Ext^1_{\cA}(\cZ_2,\Phi) = 2.   
$$
Thus,  $j^*$ is injective and so $j^*([\Phi_2])$ is a non-trivial class in $\Ext^1_{\cA,[2]}(\cZ_2,\Phi)$.  Schoof \cite[p.~5]{Sch4} describes an extension $\Upsilon$ of $\cZ_2$ by $\Phi$ such that $2 \Upsilon \ne 0$, so $[\Upsilon]$ and $j^*([\Phi_2])$ generate $\Ext^1_{\cA}(\cZ_2,\Phi)$.  But $[\Upsilon]$ is not in the image of $j^*$ when $p \equiv 9 \mymod{16}$ by \cite[Prop.~2.6]{Sch4}.  Hence $\Ext^1_{\cA}(\Phi,\Phi)$ has order 2, generated by $[\Phi_2]$.  According to Lemma \ref{VaVb}, any $(2,\Phi)$-bloc $V_2$ also is a non-trivial extension of $\Phi$ by $\Phi$, so $[V_2] = [\Phi_2]$ and $V_2 \simeq \Phi_2$.
\end{proof}

\begin{prop}  \label{PhiPhid}
Assume that $p \equiv 9 \mymod{16}$ is in {\bf P1} and $2d+2 \le h_2$.  If $V_d$ is a $(d,\Phi)$-bloc, then $V_d$ is isomorphic to $\Phi_d$.  
\begin{enumerate}[{\rm i)}]
\item If $2d+4 \le h_2$, then $\Ext^1_\cA(\Phi,\Phi_d) = \F_2$ and the non-trivial extension class is represented by 
$
0 \to \Phi_d \to \Phi_{d+1}  \to \Phi \to 0.
$ 
\vspace{2 pt}
\item If $2d+2 = h_2$, then $\Ext^1_{\cA,[2]}(\Phi,\Phi_d) = 0$.
\end{enumerate}
\end{prop}

\begin{proof} 
Since $2d+2 \le h_2$, existence of a $(d,\Phi)$-bloc $V_d$ is guaranteed.  Also, $h_2$ is a multiple of 8 when $p$ is in {\bf P1}.  Let $V_j$ be a $(j,\Phi)$-bloc in the filtration in Lemma \ref{VaVb}.  We have a non-trivial extension
\begin{equation}  \label{BlocExt}
0 \to V_{d-1} \to V_d \to \Phi \to 0
\end{equation} 
and $\Hom_\cA(\Phi,V_j) \simeq \F_2$.  The long exact sequence for $\Hom_\cA(\Phi,-)$ applied to \eqref{BlocExt} therefore gives
\begin{equation} \label{PhiByPhi}
\begin{array}{l}  
  \Ext^1_\cA(\Phi,V_{d-1})  \leftarrow \F_2  \leftarrow \F_2 \leftarrow \F_2 \leftarrow 0 \\
 \hspace{37 pt}  \downarrow  \\
 \hspace{4 pt} \Ext^1_\cA(\Phi,V_d) \xrightarrow{f} \Ext^1_\cA(\Phi,\Phi) \to \dots 
\end{array}
\end{equation}

Suppose that $2d+4 \le h_2$, so $\Phi_{d+1}$ exists.  As an induction hypothesis, assume that for $2 \le j \le d$, we have $\Ext^1_\cA(\Phi,\Phi_{j-1}) \simeq \F_2$, generated by $[\Phi_j]$.  This holds for $j=2$ by Proposition \ref{SchoofDim}.  Since $[V_j]$ is not trivial in $\Ext^1_\cA(\Phi,\Phi_{j-1})$, it follows that $[V_j] = [\Phi_j]$ and so $V_j \simeq \Phi_j$.

We can therefore replace $V_{d-1}$ by $\Phi_{d-1}$ and $V_d$ by $\Phi_d$ in \eqref{PhiByPhi}, to deduce that $f$ is an injection
$$
f\!: \,  \Ext^1_\cA(\Phi,\Phi_d) \to \Ext^1_\cA(\Phi,\Phi) \simeq \F_2.
$$
Hence the non-trivial class $[\Phi_{d+1}]$ in $\Ext^1_\cA(\Phi,\Phi_d)$ generates $\Ext^1_\cA(\Phi,\Phi_d) \simeq \F_2$. This completes the induction argument when $j = d+1$, so (i) is proved.

Next, assume that $2d+2 = h_2$ for existence of $\Phi_d$ and that 
$$
0 \to \Phi_d \to V \xrightarrow{g} \Phi \to 0
$$
represents a class in $\Ext^1_\cA(\Phi,\Phi_d)$ with $2V = 0$.  Recall that $\Phi^m \simeq \Mu_2$ over $\Z_2$.  Taking the $\cD_\lambda(M/\Q)$-modules associated to the connected components at $\lambda$ leads to the exact sequence
\begin{equation} \label{2dSeq}
0 \to \Phi_d^m(M,\lambda) \to V^m(M,\lambda) \xrightarrow{g} \Mu_2 \to 0.
\end{equation}
Hence we can find a point $x$ in $V^m(M,\lambda)$ such that $g(x)$ generates $\Mu_2$.   If $y = (\sigma_v-1)(x)$, then $\{g(y), g(x)\}$ is an $\F_2$-basis for $\Phi(M)$.  Furthermore, $\sigma_\lambda$ acts trivially on $g(y)$.   By Lemma \ref{InertiaOnVm}, $(\sigma_\lambda-1)(V(M))$ is contained in $V^m(M,\lambda)$.    Using the basis for $V^m(M,\lambda)$ in Proposition \ref{BlocProp}, we can therefore write
$$
(\sigma_\lambda-1)(y) = \alpha_1 w_2 + \alpha_2 w_4 + \dots + \alpha_d w_{2d}.
$$
Replace $x$ and $y$ by
$$
x' = x+\sum_{j=1}^{d-1} \alpha_j w_{2j+2} \quad \text{and} \quad  y' =( \sigma_v-1)(x') = y+ \sum_{j=1}^{d-1} \alpha_j w_{2j+1},
$$ 
respectively, to arrange that $(\sigma_\lambda-1)(y') = \alpha_d w_{2d}$.   

Suppose that $\alpha_d = 0$.  Then $\lr{x',y'}$ is a Galois submodule of $V$ with exponent 2 and field of points $\Q(\sqrt{p})$.  It follow that the corresponding closed subgroup scheme $U$ of $V$ is isomorphic to $\Phi$ and so the sequence \eqref{2dSeq} splits.  Hence $[V] = 0$ in $\Ext^1_\cA(\Phi,\Phi_d)$.

Suppose instead that $\alpha_d = 1$.  Then the Galois module $V(M) = \cR x'$ is cyclic and so $V$ is a $(d+1,\Phi)$-bloc.  But then $2d+4 \le h_2$ by Theorem \ref{dPhiThm} and we have a contradiction.  Thus, item (ii) holds.
\end{proof}

\section{Application to abelian varieties} \label{AbVar}

Let $\ell$ be a rational prime, $F$ a number field and $A_{/F}$ an abelian variety with good reduction outside $\ell$. Then $\Mu_\ell$ is contained in the $\ell$-division field $F(A[\ell])$ by \cite[Lemma 1]{BK1}.  We say $A$ is {\em $\ell$-prosaic} if it is semistable and $F(A[\ell])$ is an $\ell$-extension of $F(\Mu_\ell)$.  This also generalizes  the notion of {\em heavenly} abelian variety in \cite{RT}.  After Proposition \ref{IsogenyInvariant}, we only consider $\ell = 2$.  The main results of this section are in Theorems \ref{1Mod8} and \ref{A[2]bloc}.

\begin{prop}  \label{IsogenyInvariant}
The $\ell$-prosaic property of abelian varieties is an isogeny invariant.
\end{prop}

\begin{proof}
Let $\varphi\!: A \to A'$ be an isogeny of abelian varieties with quasi-inverse $\varphi'\!: \, A' \to A$, so that $\varphi \varphi'$ and $\varphi' \varphi$ are multiplication by $\deg \varphi$ on $A'$ and on $A$, respectively.  Let $\ell^m$ be the power of $\ell$ in the exponent of $\deg \varphi$.  If $P'$ is in $A'[\ell^n]$, there is a pre-image $P$ in $A[\ell^{n+m}]$ satisfying $\varphi(P) = P'$.  Hence, $F(A'[\ell^\infty])$ is contained in $F(A[\ell^\infty])$.  Use $\varphi'$ for the reverse inclusion to conclude that $ F(A'[\ell^\infty]) =F(A[\ell^\infty])$.  Note that $F(A[\ell^\infty])$ is a pro-$\ell$ extension over $F(A[\ell])$.  It follows that $A$ is $\ell$-prosaic if and only if $F(A[\ell^\infty])$ is an $\ell$-extension of $F(\Mu_\ell)$, so the $\ell$-prosaic property is isogeny invariant.

By the criterion of Ogg-N\'{e}ron-Shafarevich, good reduction is preserved under isogeny.  For bad primes $\gq$, Grothendieck's semistability criterion \cite[Prop.\! 5.13, p.\! 70]{Gro} applies.  Then $A$ is semistable at $\gq$ if and only if the inertia group at each place $\gQ$ over $\gq$ in $\Gal(F_q(A[\ell^\infty])/F_\gq)$ is tame, pro-cyclic and a generator $\sigma$ satisfies $(\sigma-1)^2 = 0$ on $A[\ell^\infty]$.  Since the isogeny $\varphi$ induces an isomorphism of Tate vector spaces 
$$
\varphi_*\!: \, V_\ell(A_{/F_\gq}) \xrightarrow{\sim}  V_\ell(A'_{/F_\gq}),
$$
semistability holds for $A$ if and only if it also holds for $A'$.  
\end{proof}

We now return to prosaic abelian varieties over $\Q$.  When $A$ is semistable and $F=\Q$, all geometric endomorphisms are defined over $\Q$ by \cite{Rib1}.

\begin{lem} \label{MultMirage}
If $A_{/\Q}$ is prosaic and $\End A=\Z$, then there is a $\Q$-isogenous abelian variety $B$ such that $B[2]$ has no subgroup scheme isomorphic to $\Mu_2$.
\end{lem}

\begin{proof}
For each $C$ isogenous to $A$, let $\gC(C)$ be a maximal subgroup scheme of $C[2]$ over $\Z[\frac{1}{p}]$, multiplicative over $\Z_2$.  An isogeny $\varphi\!: \, C_1 \to C_2$ preserves this multiplicative property, so $\varphi(\gC(C_1))$ is contained in $\gC(C_2)$.    

If $\gC(C) \ne 0$ for all $C$, then we can form a chain of $\Q$-isogenies 
$$
C_1 \to \dots \to C_j \xrightarrow{\varphi_j} C_{j+1} \to \dots
$$ 
with non-trivial $\ker \varphi_j = \gC(C_j)$ for all $j$.  By a theorem of Faltings, some $C_j$ and $C_k$ with are isomorphic for $j < k$.  The composition of the corresponding chain of isogenies is an endomorphism $\theta$ of $B = C_j$ whose kernel is entirely multiplicative over $\Z_2$.  Since $\End A = \Z$, we have $\theta=2^r$ for some $r > 0$. But the multiplicative and \'etale constituents of $B[2^r]$ have the same size \cite[p.147]{Mum}, so this is a contradiction.
\end{proof}

\begin{Rem}
The construction in the proof of Lemma \ref{MultMirage} is an example of a {\em mirage}, used in several different contexts in \cite[\S6]{BK3}.
\end{Rem}

Let $q$ range over the distinct primes dividing an odd integer $N$ and set
$$
\Omega(N) = \sum_{q \vert N} \ord_q(N) \quad \text{and} \quad \Omega_2(N) = \sum_{q \vert N, \, q \equiv \pm 1 \mymod{8}} \ord_q(N).
$$

\begin{theo} \label{1Mod8}
Let $A_{/\Q}$ be a prosaic abelian variety with odd conductor $N$ and dimension $g$.  Then
\begin{equation} \label{GenBd}
2g \le \Omega(N) + \Omega_2(N). 
\end{equation}
If, in addition, $A$ is good outside one prime $p$, then $p \equiv 1 \mymod{8}$ and $N = p^g$, that is, $A$ is totally toroidal at $p$.
\end{theo}

\begin{proof}
The conductor exponent $\ord_q(N)$ at a bad prime $q$ is the toroidal dimension of $A$ at $q$ and is at most $g$.   According to the {\em general bound} in \cite[Thm.~5.3, Cor.~5.4]{BK3}, if $A$ is prosaic, then \eqref{GenBd} holds.  If $N = p^k$ is a power of one prime $p$, then $p \equiv 1 \mymod{4}$ by \cite[Prop.~5]{BK1} and so \eqref{GenBd} gives
$$
2k \le 2g \le \begin{cases}  2k &\text{if } p \equiv 1 \mymod{8}, \\  
                                                       \hfill k &\text{if } p \equiv 5 \mymod{8}.  \end{cases}
$$
Hence $k = g$ and $p \equiv 1 \mymod{8}$.    
\end{proof}

For the rest of this section, let $A$ satisfy \eqref{AIntro}, with $p \equiv 1 \mymod{8}$ and $\End  A=\Z$. We fix some notation regarding the Tate module $T_2(A)$.

Write $L_n = \Q(A[2^n])$ for the field generated by the points of $A[2^n]$ and $L_\infty = \cup \, L_n$.  Fix a prime $\Lambda$ over 2 in $L_\infty$ and let $\lambda_n$ be its restriction of $L_n$.  See Notation \ref{Vm} for the definition of the multiplicative Galois module associated to the decomposition group $\cD_{\lambda_n}(L_n/\Q)$, namely $\cM_n =  A[2^n]^m(L_n,\lambda_n)$.  Let 
$$
\cM_\infty = T_2(A)^m(L_\infty, \Lambda) = \lim_{\leftarrow} \cM_n.
$$
Let $\zeta_\infty = \displaystyle{\lim_{\leftarrow} \zeta_n}$ in $T_2(A)$, where $\zeta_n$ generates $\Mu_{2^n}$ in $L_n$ and $\zeta_{n+1}^2 = \zeta_n$.  Choose $\sigma_\Lambda$ and $\tau_\Lambda$ in the inertia group $\cI_\Lambda(L_\infty/\Q)$ satisfying $\sigma_\Lambda(\zeta_\infty) = \zeta_\infty^{-1}$ and $\tau_\Lambda(\zeta_\infty) = \zeta_\infty^{5}$.  Then $\sigma_\Lambda$ and $\tau_\Lambda$ act trivially on the \'{e}tale quotient $T_2(A)/\cM_\infty$.  By duality, they act on $\cM_\infty$ as they do on $\zeta_\infty$.  It follows that $\sigma_\Lambda$ is an involution on $T_2(A)$.  

Fix a pro-2 generator $\sigma_\gP$ of the inertia group $\cI_\gP(L_\infty/\Q)$ at a prime $\gP$ over $p$.  By Grothendieck's criterion for semistability, $(\sigma_\gP-1)^2 = 0$ on $T_2(A)$.  Because $L_\infty$ is unramified outside $\{2,p,\infty\}$, the maximal elementary 2-extension of $\Q$ in $L_\infty$ is $\Q(\sqrt{p},\zeta_8)$.  Hence $G_\infty = \lr{\sigma_\gP, \sigma_\Lambda, \tau_\Lambda}$, by the Burnside basis theorem.  Write $\Z_2[[G_\infty]]$ for the completed group ring of $G_\infty$ over $\Z_2$. 

Recall that $M$ is the maximal 2-primary ray class field over $K$ of conductor $2\cO_K$.  There is little risk of confusion with the notation in \S\ref{Arith} if we write $\sigma_\lambda$ for the restriction of $\sigma_\Lambda$ and $\sigma_v$ for the restriction of $\sigma_\gP$ to $L_\infty \cap M$.  Also, $\tau_\Lambda$ is trivial on $L_\infty \cap M$.
 
\begin{Not} \label{Isogenies}
If $t$ is in $T_2(A)$, write $t_n$ for its projection to $A[2^n] = T_2(A)/2^n T_2(A)$.  Fix a choice of $x$ in $\cM_\infty$, such that $x_1 \ne 0$.  For $n \ge 1$, let $V_n$ denote the closed subgroup scheme of $A[2^n]$ whose Galois module is $\Z_2[[G_\infty]] x_n$ and let $f_n\!: \, A \to A_n = A/V_n$ be the natural $\Q$-isogeny with kernel $V_n$. When $n'> n,$ the inclusion $V_n \to V_{n'}$ leads to the induced exact sequence 
\begin{equation} \label{psi}
0\to V_{n'}/V_n \to A_n \xrightarrow{\psi_{n,n'}} A_{n'} \to 0
\end{equation} 
with $f_{n'} = \psi_{n,n'} \, f_n$.  By convention, $A_0 = A$ and $V_0 = 0$.
\end{Not}

For $n \ge 0$, $\kappa_n=V_{n+1}/V_n $ has exponent 2 in $A_n = A/V_n$, so $\Q(\kappa_n)$ is contained in $M$ by Proposition \ref{FGField}.  Thus, $\Z_2[[G_\infty]]$ acts on $\kappa_n(M)$ via $\cR = \F_2[\sigma_v,\sigma_\lambda]$ and $\kappa_n(M)= \cR \, f_n(x_{n+1})$ is a cyclic Galois module.  Recall that a group scheme in $\cA$ is balanced if it has a filtration with successive quotients isomorphic to $\Phi$.   

\begin{lem} \label{TowerLemma}
If $A$ satisfies \text{\eqref{AIntro}} and $A[2]$ is balanced, then:
\begin{enumerate}[\rm i)]
\item $\kappa_n$ is a $(d_n,\Phi)$-bloc in $A_n[2]$ for some $d_n$, with $1 \le d_n \le g$ and $V_{n+1}$ is balanced.   \vspace{2 pt}
  
\item $A_n[2]$ is balanced.   \vspace{2 pt}

\item $\vv{V_n} = 2^{2(d_0 + \dots + d_{n-1})}$.  \vspace{2 pt}

\item If $\theta x_{n+k}$ is in $A[2^n]$ for some integer $\theta$, then $2^k$ divides $\theta$.
\end{enumerate}
\end{lem}

\begin{proof}
The filtration of $A[2^n]$ by copies of $A[2]$ shows that $A[2^n]$ also is balanced.  According to the equivalent conditions in Proposition \ref{Phi1}, $A[2^n]$ therefore has no $\Mu_2$-subgroup scheme.  But $V_n$ is contained in $A[2^n]$, so the same holds for $V_n$ for all $n \ge 1$.  We validate (i) for $n \ge 0$ by induction.  In particular, (i) holds for $\kappa_0 = V_1$ by Proposition \ref{BlocProp}.  

Since $k_{n+1} = V_{n+2}/V_{n+1}$, the induction hypothesis and Lemma \ref{NoMu} imply that $\kappa_{n+1}$ has no subgroup scheme isomorphic to $\Mu_2$.  By construction, $\kappa_{n+1}(M)$ is a cyclic Galois module.  Thus $k_{n+1}$ also is a $\Phi$-bloc by Proposition \ref{BlocProp} and so $\kappa_{n+1}$ is filtered by copies of $\Phi$.  But $V_{n+2}$ is an extension of $\kappa_{n+1}$ by $V_{n+1}$, so $V_{n+2}$ is filtered by copies of $\Phi$.  Hence $V_{n+2}$ is balanced.  This completes the induction argument for item (i).  To finish the proof of (i), note that $d_n \le g$ because
$$
2d_n = \dim_{\F_2} \kappa_n \le \dim_{\F_2} A_n[2] = 2g.
$$

To verify (ii), it suffices to show that $A_n[2]$ has no $\Mu_2$-subgroup scheme, since $\vv{A_n[2]^m} = \vv{A_n[2]^{et}}$.  The case $n = 0$ is given.  Suppose that $A_n[2]$ has no $\Mu_2$-subgroup scheme.  Then there is no $\Mu_2$-subgroup scheme in $A_n[4]$.  Since $\kappa_n$ is filtered by copies of $\Phi$, Lemma \ref{NoMu} implies that there is no $\Mu_2$-subgroup scheme in the following image of $\psi_{n,n+1}$:
$$
0 \to \kappa_n \to A_n[4] \xrightarrow{\psi_{n,n+1}} \psi_{n,n+1}(A_n[4]) \to 0.
$$  
But $A_{n+1}[2]$ is contained in $\psi_{n,n+1}(A_n[4])$, so there is no $\Mu_2$-subgroup scheme in $A_{n+1}[2]$.  For item (iii), we have 
$$
\vv{V_n} = \vv{\kappa_1} \cdots \vv{\kappa_{n-1}} = 2^{2(d_0 + \dots + d_{n-1})}.
$$   
Finally, if $\theta x_{n+k}$ is in $A[2^n] = 2^k A[2^{n+k}]$, then $\theta$ is a multiple of $2^k$ because $x$ is not in $2T_2(A)$, by assumption.  Hence (iv) holds.
\end{proof}

Refer to the classification of primes in {\rm \eqref{P1P3}}.  By assumption, $p \equiv 1 \mymod{8}$, so $4 \vert h_2$.  The next result also holds for $g = 1$, but more precise information is in \cite{Neu,Set}.

\begin{theo}  \label{A[2]bloc}
Assume that $A$ satisfies {\rm \eqref{AIntro}}, $g \ge 2$ and $\End (A) = \Z$.  Then $2g+2 \le h_2$, $B[2]$ is a $(g,\Phi)$-bloc for some $\Q$-isogenous abelian variety $B$ and $p$ is in {\bf P1}.   In addition, if $2g+4 \le h_2$, then $p$ is in ${\bf P1^*}$.  In particular, if $g=2$, then $p$ is in ${\bf P1^*}$.   
\end{theo}

\begin{proof}
By Lemma \ref{MultMirage} and possible relabeling, assume that $A[2]$ has no subgroup scheme isomorphic to $\Mu_2$, so $A[2]$ is balanced.  For the varieties $A_j$ isogenous to $A$ in Notation \ref{Isogenies}, the theorem of Faltings implies that there is an isomorphism $\iota\!: \, A_{n'} \xrightarrow{\sim} A_n$ with $n' > n \ge 0$.  Then  $\theta=\iota \psi_{n,n'}$ is an endomorphism of  $B = A_n.$ By assumption, $A$ (so also $B$) has only trivial endomorphisms, so $\theta$ is multiplication by an integer.  Since $x_{n'}$ is in $V_{n'} = \ker f_{n'}$ by construction, we find that
$$
 f_n(\theta x_{n'}) = \theta f_n(x_{n'})  = \iota \psi_{n,n'} f_n(x_{n'}) = \iota f_{n'}(x_{n'}) = 0,
$$
so $\theta x_{n'}$ is in $\ker f_n = V_n \subseteq A[2^n]$.   By Lemma \ref{TowerLemma}(iv), $2^k$ divides $\theta$ for $k = n'-n$ and therefore $\vv{\ker \theta} \ge 2^{2kg}$.   Lemma \ref{TowerLemma} also implies that
\begin{equation} \label{kerTheta}
\vv{\ker \theta} = \vv{\ker \psi_{n,n'}} = \vv{V_{n'}}/ \vv{V_n} = 2^{2(d_n+ \dots + d_{n'-1})}.
\end{equation}
Hence $d_j = g$ for all $j =n, \dots, n'-1$ and so $B[2] = \kappa_n$ is a $(g,\Phi)$-bloc.  Furthermore, $2g+2 \le h_2$ and $p$ is in {\bf P1} by Theorem \ref{dPhiThm}.  

Since $B[2]$ is a $(g,\Phi)$-bloc, there is a closed subgroup scheme $\kappa$ of $B[2]$ isomorphic to $\Phi$.  The closed subgroup scheme $U$ of $B[4]$ associated to the Galois module 
$
\{y \in B[4] \, : \,  2y \in \kappa(M) \}
$
is a non-trivial extension of $\kappa \simeq \Phi$ by $B[2]$:
$$
0 \to B[2] \to U \xrightarrow{2} \kappa \to 0.
$$
Assume that $2g+4 \le h_2$.  Since $U$ has exponent 4, existence of this extension contradicts Proposition \ref{PhiPhid}(i) unless $p \equiv 1 \mymod{16}$.  But we already know that $p$ is in {\bf P1}, so $p$ is in ${\bf P1^*}$ by \eqref{RemP1*}.  The inequality $2g+4 \le h_2$ automatically holds when $g = 2$ because $8 \vert h_2$.
\end{proof}

\begin{Rem} \label{RMExamples}
The database \cite{LMF} provides the splitting of the space of newforms $S_2(p)$ into Galois orbits for primes $p < 10^6$.  An orbit of size $g$ corresponds to a simple, semistable abelian subvariety $A$ of $J_0(p)$ of dimension $g$ with endomorphisms by an order $\go$ in a totally real field $L$ of degree $g$.  According to \cite[Ch. 2, \S17]{Maz}, if $A$ is prosaic, then it is isogenous to a factor of the Eisenstein quotient of $J_0(p)$.  Upon examining such splittings with $2 \le g \le 20$ and $p< 10^6$, we found no prosaic examples, with the following three exceptions, up to isogeny:
\renewcommand{\arraystretch}{1.2}   
$$
\begin{array}{| c || c | c | c | c | }
\hline
p=a^2+16b^2 & a, b & g = \dim A &  \text{RM} & h_2 \\
\hline
\, 41 & 5,1 & 3 & \text{cubic} & 8  \\
\hline
113  & 7, 2 & 2 & \sqrt{3} & 8  \\
\hline
1201  & 25, 3 & 2 & \sqrt{2} & 16  \\
\hline
\end{array}
$$
\renewcommand{\arraystretch}{1}

\noindent The cubic field for $J_0(41)$ has discriminant $4 \cdot 37$ \cite{Web}.

To test the prosaic property, we used the Fourier coefficients $a_\ell$ of the cusp forms given by the database for odd primes $\ell < 100$.  Let $\gp$ be a prime over 2 in $\go$, let $F = \Q(A[\gp])$ and let
$
\bar{\rho}_\gp\!: \, \Gal(F/\Q) \to \GL_2(\go/\gp)
$
be the corresponding mod $\gp$ representation. If $\varphi_\ell$ is a Frobenius at $\ell \nmid 2N$, then $\Tr(\bar{\rho}_\gp(\varphi_\ell)) \equiv a_\ell \pmod{\gp}$.  If there is some $\ell$ and some prime $\gp$ with $a_\ell$ not in $\gp$, then $A$ is not prosaic.
For the three exceptional cases above, there is unique prime $\gp$ over 2 in $\go$ and $\go/\gp \simeq \F_2$.  These examples are prosaic, thanks to the following observation.
\end{Rem}

\begin{lem} 
Let $F/\Q$ be a Galois extension, unramified outside $\{2, p, \infty\}$, with $p \equiv 1\mymod{8}$ and $[F:\Q]=3$ or $6$.  Assume that
\begin{enumerate}[{\rm i)}] 
\item the ramification index of each prime $\gp$ over $p$ in $F$ divides $2$ and
\item the primes $\gq$ over $2$ in $F$ satisfy the Fontaine bound \eqref{FontBd} with $n=1$.
\end{enumerate} 
 Then the class number of $\Q(\sqrt{p})$ or $\Q(\sqrt{-p})$ is a multiple of $3$.  
\end{lem}

\begin{proof}  
The fixed field $k$ of the cubic subgroup of $\Gal(F/\Q)$ is at most quadratic over $\Q$.  Since finite places outside $2p$ are unramified, (ii) implies that $k$  is contained in $\Q(i,\sqrt{p})$.  By (i), $F/k$ is unramified at all $\gp \vert p$.  At each prime over 2 in $k$, the residue field is $\F_2$, so local class field theory shows that $F/k$ also must be unramified at each $\gq$.  It follows that $k = \Q(\sqrt{\pm p})$, with class number a multiple of 3. 
\end{proof} 

\section{Prosaic families, Richelot isogenies and mild reduction}  \label{Examples}

Theorem \ref{1Mod8} gives a lower bound $g$ for the number of prime factors in the conductor of a prosaic abelian variety $A$ of dimension $g$. This bound is attained when $g=2$: we exhibit several such families of genus 2 curves 
\begin{equation}  \label{curveC}
C\!: \, y^2 + Q(x) \, y = P(x)
\end{equation}
whose Jacobian $A=J(C)$ is prosaic and typical.  They provide, under the Schinzel hypothesis, infinitely many surfaces  of conductor $pq$ for distinct primes $p,\,q$.

The splitting field $L$ of $F(x) = Q(x)^2+4P(x)$ is the 2-division field of $A$.  To ensure that $\Gal(L/\Q)$ is a 2-group, we assume a factorization 
\begin{equation} \label{FacF}
F(x) = f(x) \, h(x) \, \ov{h}(x) \quad \text{with} \quad G(x) = h(x) \, \ov{h}(x),
\end{equation} 
where $\deg f = 1$ or 2 and $h$ and $\ov{h}$ are conjugate quadratics over $\Z[i]$.  According to \cite[\S2]{Liu1}, the discriminant of the model $\eqref{curveC}$ is 
\begin{equation} \label{discC}
\Delta_C = \begin{cases} c_0^2 \Delta_F/2^{12}  &\text{if } \deg F = 5,\\[2 pt]
                                        \hfill \Delta_F/2^{12}  &\text{if } \deg F = 6,                      
                   \end{cases}
\end{equation}
where $\Delta_F$ is the discriminant and $c_0$ is the leading coefficient of $F$.   The discriminants and resultants of the polynomials in \eqref{FacF} are related by
\begin{equation}\label{discF}
 \Delta_F = \Delta_f \, \Delta_h \, \Delta_{\ov{h}} \, \res(h,\ov{h})^2 \, \res(f, G)^2.
\end{equation}

 The divisor classes of order 2 in $J(C)$ corresponding to $h(x) = 0$ and $\bar{h}(x) = 0$ generate a rational $(2,2)$-subgroup $\kappa$ of $A[2]$ which is totally isotropic for the Weil pairing and serves as the kernel of a {\em Richelot isogeny} to a Jacobian $J(C')$.  For each family of examples, we give a model for $C'$, utilizing the factorization \eqref{FacF}.  For more details about Richelot isogenies and additional references, see \cite[Ch.\! 9]{CF} or \cite[Ch.\! 8]{Smi}.  In \cite{CF}, the factorization of $F(x)$ into three quadratic polynomials is assumed to occur over the ground field.  If two of the quadratic factors are conjugate over $\Q(\sqrt{d})$, the curve $y^2=F'(x)$ in \cite{CF} must be changed to $y^2=dF'(x)$.

Some of our families involve primes $p$ of {\em mild reduction}, namely:  $C$ has bad reduction at $p$, but $J(C)$ reduces to a product of two elliptic curves over $\ov{\F}_p$. As in \cite{BK3}, there is a monic quadratic polynomial $s$ in $\Z_p[x]$ with separable reduction over $\F_p$, such that $P$ and $Q$ in \eqref{curveC} have the form 
\begin{equation} \label{LiuModel}
\begin{array}{l l l}
Q = m(sa_1+ma_3), &  
P = m(a_0s^3 + m a_2s^2 + m^2 a_4 s + m^3 a_6),  \\[3 pt]
m \in p\Z_p, \hspace{7 pt} a_0 \in \Z_p^\times,  &a_{j} \in \Z_p[x], \hspace{3 pt} \deg a_{j} = 1 \text{ for } j = 1, \dots, 6.
\end{array}
\end{equation}
Let $r_1$ and $r_2$ be the roots of $s(x)$, possibly in an unramified quadratic ring over $\Z_p$.  If, for $j = 1,2$, each of the substitutions $x =r_j + m X$, $y = m^2 Y$ reveals an elliptic curve in the naive reduction, then $p$ is a mild prime. Such a model is non-minimal at $p$.  Minimizing the discriminant of $C$ is accomplished by either of these changes of variables when the roots of $s$ are in $\Z_p$.  See \cite{Liu2} for a discussion of minimal models and reduction types. 

When $z^{g+1} Q(x/z)$ is separable over $\F_2$, there is no biconnected subquotient in $A[2]$ over $\Z_2$ and so a natural choice for our first family is $Q(x) = x^2+x+b$.  Let $f(x) = 4x+4a+1$ and $h(x) = Q(x) - 2i$, as a simple way to satisfy the congruence
\begin{equation} \label{mod4}
(4x+4a+1) \, h(x) \, \ov{h}(x) = F(x) \equiv Q(x)^2 \pmod{4 \, \Z[x]}.
\end{equation}
Then $P(x)$ is defined by $4P = F - Q^2$.

Write $\core(n)$ for the product of the distinct prime factors of $n$. 
\begin{prop} \label{AB1}
For integers $a,b$, let $Q(x) = x^2+x+b$ and let $J(C)$ be the Jacobian of the curve {\rm\eqref{curveC}}, with 
\begin{eqnarray}\label{eqn1}
F(x) = (4x+4a+1) \, (Q(x)-2i) \, (Q(x)+2i).
\end{eqnarray}
Set $m = (4b-1)^2+64$ and $n = ((4a-1)^2 + 16 b-4)^2+1024$. 
If $\gcd(m,n)=1$, then $J(C)$ is prosaic, with conductor equal to $\core(mn)$.
\end{prop}

\begin{proof}
If $h(x) = Q(x) - 2i$ and $G(x) = h(x) \, \bar{h}(x)$, then 
$$
\res(h,\bar{h}) = -16, \quad \gm = \Delta_h = 1-4b+8i, \quad \hspace{4 pt}
\res(4x+4a+1,G) = n.  
$$
Since $m = \gm \, \ov{\gm}$, formulas \eqref{discC} and \eqref{discF} give $\Delta_F = 2^8mn^2$ and $\Delta_C = mn^2$. 

Assume that $\gcd(m,n) = 1$ and let $\ell$ be a prime of bad reduction for $C$.   Necessarily, $\ell$ splits in $\Z[i]$ and a choice of prime $\pi$ over $\ell$ in $\Z[i]$ determines the image of $F(x)$ in $\F_\ell[x]$.   

Case 1:  $\ell \vert m$.  By symmetry, we may assume that $\pi \vert \gm$, so $h$ has a double root modulo $\pi$, say $x=\alpha$ in $\F_\ell$.  Since $\pi$ does not divide
$$
\gcd(\Delta_h, \Delta_{\bar{h}}) = \gcd(1-4b+8i,1-4b-8i) = 1,
$$
$\alpha$ is not a root of $\bar{h}$.  But $m$ and $n$ are relatively prime, so $\pi$ does not divide $\res(4x+4a+1,G) = n$.  Thus, $\alpha$ is not a root of $4x+4a+1$ over $\F_\ell$.  Hence $\alpha$ is a node of $C$ over $\F_\ell$.

\vspace{2 pt}

Case 2:  $\ell \vert n$.  Then $\beta = -(4a+1)/4$ is a common root of  $4x+4a+1$ and $G(x)$ modulo $\ell$.  Since $\ell$ does not divide 
$$
\Delta_G  = \Delta_h \, \Delta_{\bar{h}} \, \res(h,\bar{h})^2 = 256m,
$$
$\beta$ is a simple root of $G$  modulo $\ell$ and therefore $\beta$ is a node of $C$ over $\F_\ell$.

Hence $J(C)$ is semistable, with squarefree conductor $\core(mn)$. The splitting field of $F$ is $L=\Q(i, \sqrt{m}, \sqrt{1-4b+8i})$. It is a $\cD_4$-field unless $b=4$, in which case $L=\Q(i)$ and $n$ is a reducible polynomial of $a$.
\end{proof}

\begin{Rem}
The Richelot-isogenous curve for $J(C)$ in Proposition \ref{AB1} is isomorphic to $C'\!: \, y^2=P'(x)$, with $Q'=0$,  $P'=x h'\bar{h}'$ and
$$
h' = x^2+(1-4a)x+(1-4b) -8i.
$$
The discriminant of $C'$ is $\Delta_{C'} = 2^{24}m^2 n$ and 2 is a mild prime for $C'$  
\end{Rem}

\begin{Ex}
Some cases of Proposition \ref{AB1} with $m$ and $n$ prime. 
{\small $$
\begin{array}{| r | r | r | r ||  r | r | r | r || r | r | r | r  || }
\hline
a&b&m&n&a&b&m&n&a&b&m&n\\
\hline 
 0& 1& 73& 1193 & -1& -1& 89& 1049& -1& 1& 73& 2393  \\
 2& -3& 233& 1033&  1& 2& 113& 2393&  -1& -3& 233& 1753\\
  -2& -4& 353& 1193 & -1& 2& 113& 3833&3& -4& 353& 3833 \\
 3& 1& 73& 18713&  3& -9& 1433& 1753& -3& -8& 1153& 2393\\
  4& -14& 3313& 1033& -3& -3& 233& 14713& 4& -1& 89& 43049\\
 \hline
\end{array}
$$}
\end{Ex}

\vspace{5 pt}

Assuming the  Schinzel hypothesis, $m$ and $n$ simultaneously are primes infinitely often.  The resulting examples cannot be isogenous to the product of elliptic curves.  Otherwise $n$  is also a Setzer-Neumann prime, i.e., of the form $v^2+64$.  But if $n = w^2 + 1024=v^2+64$ with $v,w$ in $\Z$, then 
$$
(w+32i)(w-32i)=(v+8i)(v-8i)
$$ 
violates unique factorization in $\Z[i]$. Note that $J(C)$ and $J(C')$ are semistable with squarefree conductors, so they are typical, thanks to \cite{Rib1}.

\begin{Ex} \label{Ex2}
As a variant of Proposition \ref{AB1}, let $c \equiv 1 \mymod{\text{mod } 4}$, $b = \pm 1$ and $n = (c^2 - 4bc-16)^2 + 64c^2$.  Assume that $17 \nmid n$ and let $C$ be given by \eqref{curveC}, with 
$$
Q = x^2 + b x - 1,  \quad  h = Q-2ix, \quad  F = (4x+c) h \bar{h}.
$$  Then $\Delta_C = 17 n^2$ and $J(C)$ is prosaic and typical, with conductor  $\core(17n)$.  The Richelot-isogenous curve is isomorphic to $C'\!: \,  y^2 + Q' y =P'$, where $Q'=x(x^2+x+b)$, 
$$
h'(x)= x^2 + cx + bc+4 +2c i \hspace{5 pt} \text{ and } \hspace{5 pt}F'(x)=c(x^2+4)h'(x)\bar{h}'(x). 
$$
The discriminant $\Delta_{C'} = -17^2 c^{22}$.  Primes $p \ne 17$ dividing $c$ are mild for the curve $C'$.  If $17 \vert c$ then, in the mild model at 17, exactly one of the naive reductions has a node and the other is an elliptic curve.  The following are  examples of conductor $17n$ with $n$ prime:
$$
\begin{array}{| r | r | r || r |  r | r || r | r | r || r | r | r  | }
\hline
b&c&n&b&c&n&b&c&n&b&c&n\\
\hline 
 1& -3& 601& -1 & -3& 937&1& 5&1721& -1& 5& 2441  \\
 1& -7& 6857& 1&  13& 21017& 1& 17&60521& -1&17& 134777\\
\hline
\end{array}
$$
\end{Ex}

\vspace{5 pt}

\begin{Ex}   \label{MildProsaics}
Our last family consists of genus 2 curves with prosaic, typical Jacobian and some  mild reduction.   Let $u$ be odd, $c \equiv 1 \mymod{4}$,
$$
\begin{array} {l l l}
m=64u^2+1, & &n = (c u - 16)^4 + 64 c^2, \\[3 pt]
s = x(x-1), & &h = u (x+1)^2 - s i
\end{array}
$$
and assume that $\gcd(um,n) = 1$.  If $Q= 2u (ux - x - u) s$ and 
$$
F(x) =4 u  x  (c  u  x - 8 x + 8)  h  \bar{h}, 
$$
the curve $C$ has the form \eqref{LiuModel} with $\Delta_C = (2u)^{22} \,m n^2$.  Primes dividing $2u$ are mild and the conductor of $J(C)$ is $\core(mn)$.  The Richelot-isogenous curve has a model of the form $C'\!: \, y^2 + Q' y =P'$, with
$$
Q'=c^2 x^2 (x+u), \quad h' = c x^2 - c u x + 16 x - 2 i,  \quad F' = c (c x+16) h' \bar{h}'
$$
and minimal discriminant $\Delta_{C'} = (2c)^{12}\, m^2 n$.  Primes dividing $2c$ are mild.  Some examples with $m$ and $n$ prime follow.
{\small $$
\begin{array}{| r | r | r | r ||  r | r | r | r || r | r | r | r  | }
\hline
u&c&m&n&u&c&m&n&u&c&m&n\\
\hline 
 -5& -3&1601&577 & -3& -7& 577& 3761& -7& -3& 3137& 1201  \\
 5& 5& 1601& 8161& 3& 13& 577& 290657&  -3& 5& 577& 925121\\
 \hline
\end{array}
$$}
\end{Ex}

\vspace{5 pt}

\begin{Rem} \label{1797}
Among typical prosaic Jacobian surfaces, $17\cdot97$ is the smallest conductor known to us \cite{BK3}.  The isogeny class is represented by $J(C)$, where $C$ is given by \eqref{curveC}, with
$$
\begin{array}{l l}
P = x^4 + x^3 + x^2 - 3x + 1, &\quad Q = (x+1)(x^2+x+1) \\[4 pt]
\multicolumn{2}{c}{F = (x^2 + 2x + 5)(x^2 + x - 1+2i  x)(x^2 + x - 1 - 2 i x ).}
\end{array}
$$
The Richelot-isogenous curve is isomorphic to 
$$
C': y^2=3(x^2 - 2x + 2)h'\bar{h}',
$$
where $ h'=(x^2-5x-1)+2(x^2+ x-1)i$ and $Q' = 0$.  The primes 2 and 3 are mild and $\Delta_{C'}=-6^{22}\cdot17^2\cdot97$.
Neither $C$ nor $C'$ seems to belong to a  family such as those above. 
\end{Rem}





\normalsize
\baselineskip=17pt


\vspace{5 pt}

\noindent{\bf Errata in \cite{BK2}}.
In \cite[Cor.~3.6]{BK2}, we asserted without full justification, that if an abelian variety $A$ satisfies \eqref{AIntro} then $p \equiv 1 \mymod{8}$.  The argument is completed here, in Theorem \ref{1Mod8}.  In addition, \cite[p. \!15, line $-6$]{BK2} should say that $\Q(J_0(41)[2])/\Q$ is a 2-extension and $\Q(J_0(31)[5])/\Q(\Mu_5)$ is a 5-extension, as is explained correctly, with more details, in \cite[p. \!35]{BK2}.

\vspace{5 pt}

\subsection*{Acknowledgement}
Research of the second author was partially supported by the PSC-CUNY Research Award Program.

\vfill
\end{document}